\documentclass[letterpaper]{article}
\usepackage[utf8x]{inputenc}
\usepackage[T1]{fontenc}
\usepackage[titletoc]{appendix}

\usepackage[margin=1in]{geometry}

\usepackage{amsmath,amssymb,amsthm,bbm,bm}
\usepackage{mathrsfs}
\usepackage[linesnumbered,ruled,noend]{algorithm2e}
\theoremstyle{definition}

\newtheorem{Remark}{Remark}
\newtheorem{Assumption}{Assumption}
\newtheorem{Definition}{Definition}

\newtheorem{Proposition}{Proposition}
\newtheorem{Theorem}{Theorem}

\usepackage{graphicx}
\usepackage[colorinlistoftodos]{todonotes}
\renewcommand{\Re}{\mathbb{R}}
\renewcommand{\Pr}{\mathrm{Pr}}

\newcommand{\Expectation}{\mathbb{E}}
\newcommand{\trace}{\mathtt{tr}}

\newcommand{\ScriptA}{\mathcal{A}}
\newcommand{\ScriptB}{\mathcal{B}}
\newcommand{\ScriptD}{\mathcal{D}}

\usepackage[subrefformat=parens,labelformat=parens]{subfig}
\newcommand{\minimize}{\textrm{minimize}}

\SetKwInput{KwInput}{Input}
\SetKwInput{KwOutput}{Output}

\usepackage[symbol]{footmisc}

\usepackage[breakable]{tcolorbox}

\title{Stochastic Model Predictive Control for Constrained Linear Systems Using Optimal Covariance Steering}

\author{Kazuhide Okamoto\thanks{K. Okamoto is with the School of Aerospace Engineering, Georgia Institute of Technology, Atlanta, GA 30332-0150, USA. Email: kazuhide@gatech.edu} \qquad
	Panagiotis Tsiotras\thanks{P. Tsiotras is with the School of Aerospace Engineering, and also with the Institute for Robotics and Intelligent Machines, Georgia Institute of Technology, Atlanta, GA 30332-0150, USA. Email: tsiotras@gatech.edu}}
	
\begin{document}
	\maketitle
	
\begin{abstract}
	This work develops a stochastic model predictive controller~(SMPC) for uncertain linear systems with additive Gaussian noise subject to state and control constraints.
	The proposed approach is based on the recently developed finite-horizon optimal covariance steering control theory, which steers the mean and the covariance of the system state to prescribed target values at a given terminal time.
	We call our approach covariance steering-based SMPC, or CS-SMPC.  
	We show that the proposed approach has several advantages over traditional SMPC approaches in the literature.
	Specifically, it is shown that the newly developed algorithm can deal with unbounded Gaussian additive noise while ensuring stability and recursive feasibility, and incurs lower computational cost than previous similar approaches.
	The effectiveness of the proposed CS-SMPC approach is confirmed using numerical simulations.
\end{abstract}

\section{Introduction}\label{sec:Introduction}

Model predictive control (MPC), often also referred to as receding horizon control, has been an active research topic both in academia and industry because of its ability to deal with complex constraints, while also yielding near-optimal performance.
In the standard MPC framework, one solves a finite-horizon optimal control problem and executes the first element of the computed optimal control sequence.
At the next time step, one solves another finite-horizon optimal control problem with the updated initial condition.
By doing so, MPC implicitly closes the loop and achieves stability, assuming certain additional conditions hold~\cite{mayne2000constrained}.
Several variants and extensions of MPC have been proposed in the literature, such as explicit MPC~\cite{bemporad2002model,bemporad2002explicit,besselmann2012explicit}, the hybrid MPC~\cite{zhang2015computationally}, and learning MPC~\cite{rosolia2018learning,bouffard2012learning} , while several MPC versions have been applied to a variety of engineering and industrial domains~\cite{OldJonParMor2014,GrayBor2013,DiCairano2018,del2010automotive,DiCairano2018,eren2017model}.

Since MPC is a model-based control design, deterministic MPC approaches
are susceptible to errors owing to modeling uncertainties and exogenous disturbances.
In order to overcome this difficulty, robust MPC~(RMPC) and stochastic MPC (SMPC) extensions have been developed (see e.g.,~\cite{bemporad1999robust, mesbah2016stochastic,farina2016stochastic,kouvaritakis2016book} for an extensive literature review) to deal with various forms of uncertainty.

Robust MPC approaches assume deterministic uncertainties, which lie in a given compact set.
For example, min-max MPC~\cite{raimondo2009min} computes a control command that can deal with the worst-case scenario in terms of system uncertainty. 
Another RMPC approach is the tube-MPC~\cite{langson2004robust}, which 
separates the controller to a nominal controller and a feedback controller that is proportional to the deviation from the nominal state value using 
a stabilizing state feedback gain, and thus achieves asymptotic stability to a set~\cite{mayne2005robust}.
RMPC approaches are effective against worst-case deterministic disturbances, but can be conservative in case of stochastic disturbances, 
since they ignore any knowledge about the probabilistic nature of the disturbance.  
In addition, they do not even guarantee recursive
feasibility in case of possibly unbounded disturbances.
Recursive feasibility is a crucial property that ensures that the MPC optimization problem has a solution (i.e., satisfies all the constraints) at each time step~\cite{mayne2014model,mayne2000constrained}.

In order to explicitly deal with the probability characteristics of system uncertainties, several stochastic MPC (SMPC) approaches have been developed.
As with RMPC, in SMPC feedback \emph{policies} are optimized instead of an open-loop control sequence alone.
However, SMPC abandons the worst-case point of view for an ``expected'' or ``average'' system behavior that takes into consideration the most likely disturbance (instead of the worst-case disturbance) the system may encounter in practice.
As a result, SMPC methods tend to provide better performance and they can even deal with unbounded disturbances.

Although there is no agreed consensus for classifying the numerous SMPC approaches proposed in the literature~\cite{mesbah2016stochastic}, 
the most common approaches are the so-called analytic approaches and the randomized (or scenario-based) approaches.
The former include stochastic-tube~\cite{cannon2009probabilisticConstrained,cannon2009probabilisticTubes,carvalho2014stochastic} and affine-parameterization~\cite{hokayem2012stochastic,oldewurtel2008tractable} approaches, which reformulate the cost and the probabilistic constraints in deterministic terms.
They typically assume some form of additive white Gaussian noise acting on the system, and are most closely related to the proposed CS-SMPC.
Scenario approaches such as~\cite{bernardini2009scenario,calafiore2012robust}, on the other hand,
compute expected future system behavior by generating randomly several noise realizations.
For this reason they can handle more generic systems, costs and state and control constraints.
However, their computational requirements are much higher than analytic approaches.
In addition, their feasibility and convergence properties are difficult to access.
Scenario-based approaches will not be discussed further in this work.

From the analytic approaches, 
stochastic-tube MPC~\cite{cannon2009probabilisticConstrained, cannon2009probabilisticTubes, cannon2011stochastic, kouvaritakis2010explicit,lorenzen2019stochastic,hewing2018stochastic} decomposes the system state to deterministic and random components. 
The random component is controlled using a state feedback controller with a \emph{pre-computed} stabilizing gain, and only the additional control command to steer the deterministic component is computed online.
By doing so, the stochastic-tube MPC approach succeeds in avoiding the optimization over arbitrary feedback policies. 
Although this approach reduces computational complexity, it requires
trial and error to compute a priori a state feedback gain that is not too conservative, especially when constraints are expected to be active. 

In order to overcome the off-line computation of the feedback gain, the affine parameterization SMPC approach has been proposed~\cite{hokayem2012stochastic,oldewurtel2008tractable}. 
In the affine parameterization approach, both the feedback gain and the deterministic component are design variables that have to be simultaneously optimized online. 
Figure~\ref{fig:StochasticTubeVSCSSMPC} tries to illustrate pictorially the main difference between these two approaches. 
The stochastic-tube MPC approach knows the future state uncertainty evolution a priori because the feedback gain is pre-computed, and it tries to control the mean state so that  the predicted state satisfies the given constraints at the end of the horizon.
On the other hand, the affine parameterization approach simultaneously computes the feedback gains and the open-loop control sequences so that the predicted state solution satisfies the constraints.
As a result, the affine parameterization approach leads to less conservative controllers that tend to operate closer to the boundary of the constraints, thus increasing performance.

It is known that a state feedback parameterization approach leads to a non-convex problem~\cite{primbs2009stochastic,farina2013probabilistic,farina2015approach}.
In practice, one thus relaxes the constraints to make the problem convex, which may  lead to unnecessarily conservative results.
As an alternative, in~\cite{oldewurtel2008tractable} the authors employ 
a disturbance feedback parameterization of the control policy instead, which leads to a convex problem formulation.
It has been shown that disturbance feedback parameterization
is equivalent to state feedback parameterization~\cite{goulart2006optimization}. 
The disturbance feedback parameterization approach has been extended to accommodate input hard constraints in~\cite{hokayem2009stochastic,hokayem2010stable,hokayem2012stochastic,paulson2017stochastic}, a task that 
is difficult to satisfy using stochastic tube-MPC or state feedback parameterization approaches.

\begin{figure}[htbp]
	\centering
	\subfloat[Stochastic-tube SMPC approach.\label{fig:StochasticTubeState}]{\includegraphics[width=0.5\columnwidth]{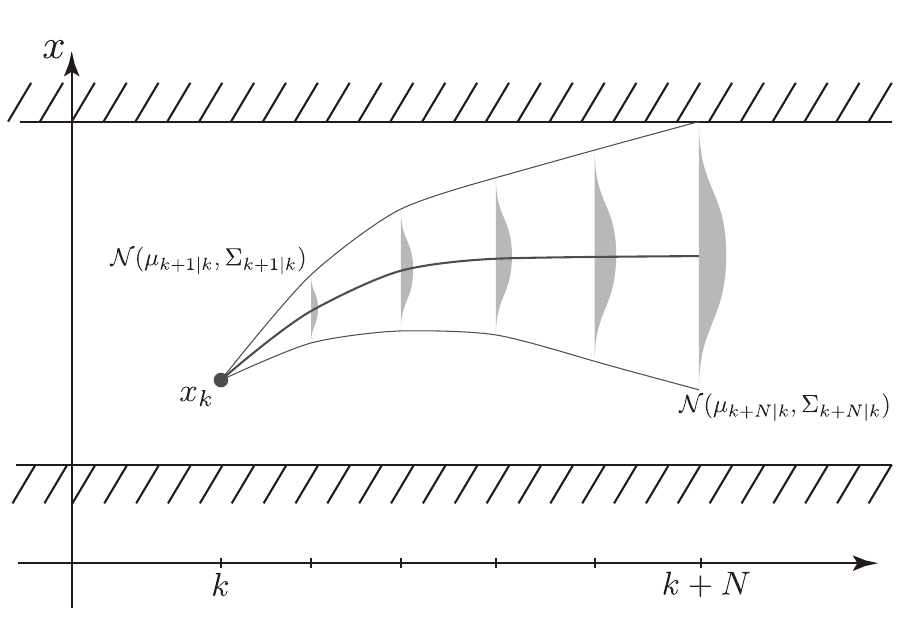}}
	\subfloat[Affine parameterization SMPC approach.\label{fig:CSSMPCState}]{\includegraphics[width=0.5\columnwidth]{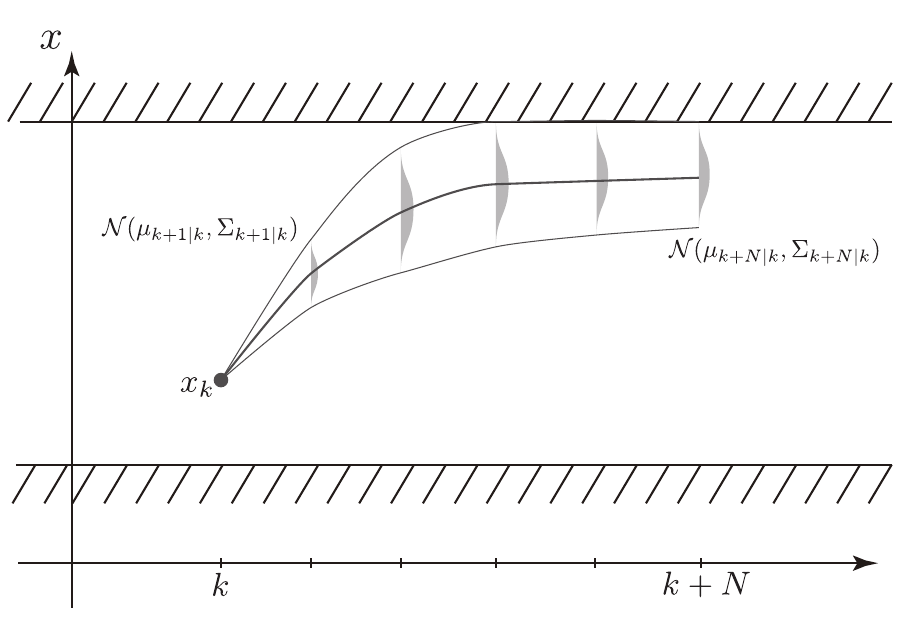}}\\
	\caption{Comparison of the state trajectory evolution by various SMPC approaches.~\label{fig:StochasticTubeVSCSSMPC}
	}
\end{figure}

Owing to the  problem stochasticity,  SMPC uses probabilistic (e.g., chance) constraints 
and imposes a maximum probability of state or input constraint violation~\cite{schwarm1999chance, ono2012joint}, 
instead of an absolute constraint violation requirement.
Satisfaction of the constraints with high probability is the price to pay for being able to explicitly handle uncertainty in the problem formulation.
When using the stochastic tube-MPC approach (and since the feedback gain is pre-computed) the chance constraints can be converted to linear inequality constraints, whereas in the affine parameterization approach (and since the feedback gains are computed online)
the chance constraints are converted to second-order cone constraints. 
Some applications of SMPC include building climate control~\cite{ma2012fast}, autonomous vehicle control~\cite{carvalho2014stochastic}, and bacterial fermentation control~\cite{paulson2017stochastic}.

As already mentioned, MPC controller design is based on a receding horizon approach, by which an optimization problem is solved repeatedly every time a new measurement of the state is available.
As a result, a key requirement for the success of an MPC design is the ability to satisfy the constraints at each iteration, a property known as
recursive feasibility.
While the issue of recursive feasibility is well understood for the case for deterministic MPC, 
showing the same for SMPC is much more challenging.
This is because, in general, it is not  possible to
enforce recursive satisfaction of the state and control constraint in the face
of \textit{unbounded} additive disturbances.
Therefore,  most works on SMPC achieve recursive feasibility and convergence by assuming a bounded probability distribution of the disturbance, e.g., a truncated Gaussian distribution.
The theory for SMPC dealing with unbounded disturbances is less developed.
Some recent work includes~\cite{paulson2017stochastic, farina2013probabilistic, paulson2015stability,KouCannon15}.
In \cite{KouCannon15} a stochastic tube-SMPC approach is proposed,
but the approach cannot handle hard input constraints as the feedback gains are computed offline.
The approach proposed in~\cite{farina2013probabilistic} devises a 
re-initialization strategy that switches between closed-loop and open-loop control  to ensure recursive feasibility.
It also uses state-feedback parameterization, which leads to the need to solve a non-convex programming problem.
Thus, these approaches may have difficulty in computing the system covariance at each time step.
In~\cite{paulson2017stochastic} the authors consider the MPC problem for stochastic
linear systems with arbitrary, possibly unbounded, disturbances subject to 
both joint state chance constraints and hard input constraints.
Contrary to the work in~\cite{farina2013probabilistic} the authors of~\cite{paulson2017stochastic} do not use a re-initialization strategy to ensure recursive feasibility and stability, 
but instead they suggest to soften the chance constraints.
In \cite{paulson2015stability} the authors address a SMPC problem for additive Gaussian process noise and time-invariant probabilistic  system uncertainty using a state feedback parameterization.
Since guaranteeing feasibility and stability under both unbounded noise and system uncertainty is very challenging, 
stability is established  for the unconstrained case only.

The main difficulty in all SMPC problems is controlling the dispersion of the trajectories owing to noise, which makes
ensuring recursive feasibility challenging, especially in case of noise with unbounded support. 
The approach we propose in this article overcomes this difficulty of controlling the state dispersion at the end of the horizon,
by utilizing the results from the newly developed finite horizon optimal covariance steering theory~\cite{okamoto2018Optimal,okamoto2018Optimalb}.
An optimal covariance steering controller steers the covariance of a stochastic linear system to a target terminal value, while minimizing a state and control expectation-dependent cost.
While infinite horizon covariance control has been researched extensively 
since the late `80s~\cite{hotz1985covariance, hotz1987covariance}, the finite horizon covariance steering case had not been investigated until very recently~\cite{chen2016optimalI,bakolas2016optimalCDC,bakolas2018finite,goldshtein2017finite}.
Specifically, and most closely related to the results in this paper, in our previous work~\cite{okamoto2018Optimal} 
we introduced state chance constraints into the optimal covariance steering problem, and used it to solve challenging path planning problems in the presence of uncertainty~\cite{okamoto2018Optimalb}.
To our knowledge, the current paper is the first work  in the literature to incorporate optimal covariance steering into the SMPC framework,
and thus deals with the issue of recursive feasibility and convergence of SMPC in a principled manner.
Since the proposed CS-SMPC approach simultaneously computes the open-loop and the feedback gains, it can be regarded as an affine parameterization approach.

The main contributions of this work are summarized as follows:
First, we introduce a new SMPC approach for systems with unbounded noise 
that takes advantage of the recent theory of covariance steering to construct the terminal constraint sets that are needed in order to ensure stability and recursive feasibility.
The use of finite-time covariance steering allows for the direct control of the state covariance at the end of the horizon, 
which encodes the specification that the state will be in a given terminal set with high probability, 
thus making the analysis for stability and recursive feasibility less conservative and
more direct and intuitive.
Second, 
we introduce a new feedback strategy for SMPC in terms of a filtered version of the disturbance noise that results in a convex problem formulation.
This new feedback parameterization yields a structured block-diagonal matrix of the gains, resulting in lower computational effort compared to the 
disturbance feedback parameterization.
Third, we elucidate the connections between optimal covariance steering and the affine-parameterized policies used in several SMPC methodolgies.
Although optimal covariance steering and the affine-parameterization SMPC approach have been developed independently from each other, they have many theoretical similarities, and by combining them together we can achieve computational efficiency along with the ability to deal with unbounded Gaussian disturbances, while also guaranteeing feasibility and stability.

The remainder of this paper is organized as follows.
Section~\ref{sec:ProblemStatement} formulates the problem and introduces the general SMPC problem setup.
In Section~\ref{sec:Preliminaries}
we review some mathematical preliminaries and briefly discuss the newly developed finite-horizon optimal covariance steering algorithm~\cite{okamoto2018Optimalb}, which is the basis for the CS-SMPC.
Section~\ref{sec:CS_SMPCDesign} introduces the proposed CS-SMPC approach, followed by the proof of recursive feasibility and guaranteed stability.
In Section~\ref{sec:Numerical Simulation} we validate the effectiveness of the CS-SMPC approach using numerical simulations.
Finally, Section~\ref{sec:Summary} summarizes this work and discusses  possible future research directions.

\section*{Notation}

The notation used in this paper is quite standard.
We denote the set of $n$-dimensional real vectors and $m\times n$ real matrices by $\Re^{n}$ and $\Re^{m\times n}$, respectively.
We use $P \succ 0$ and $P \succeq 0$ to denote the fact that the matrix $P$ is symmetric positive definite and semidefinite, respectively.
Also, we use $P \geq 0$ and $P > 0$ to denote element-wise inequalities.
$\trace(P)$ denotes the trace of the square matrix $P$, and $\mathtt{blkdiag}(P_0,\ldots,P_N)$ denotes the block-diagonal matrix with 
matrices $P_0,\ldots, P_N$.
$\|v\|$ is the 2-norm of the vector $v$ and 
$\|P\|_F$ denotes the Frobenius norm of the matrix $P$.
$I_d \in \Re^{d\times d}$ is the identity matrix of size $d$ and
$\mathcal{R}(P)$ denotes the range space of the matrix $P$.
The notation
$x \sim \mathcal{N}(\mu, \Sigma)$ indicates that the random variable $x$ is
sampled from a Gaussian distribution with mean $\mu$ and (co)variance $\Sigma$.
$\Expectation$ denotes the expectation operation and
$\Expectation_k[\cdot] = \Expectation[\cdot|x_k]$ denotes the expectation conditioned on the measured state $x_k$ at time step $k$. 
Finally, 
$\Pr(A)$ denotes probability of the event $A$, and $\Pr_k(A) = \Pr(A | x_k)$ denotes the conditional probability given $x_k$.

\section{Problem Statement\label{sec:ProblemStatement}}

In this section we formulate the general SMPC problem. 

\subsection{Problem Formulation}

We consider the following discrete-time stochastic linear time-invariant (LTI) system with additive noise,
\begin{equation}  \label{eq:SystemDynamics}
	x_{k+1} = Ax_k + Bu_k + D w_k, \quad x_0 \sim \mathcal{N}(\mu_0,\Sigma_0),
\end{equation}
where $k$ is the time-step index, $x_k \in\Re^{n_x}$ is the state, $u_k \in\Re^{n_u}$ is the control input, and $w_k \in\Re^{n_w}$ is a zero-mean independently and identically distributed (i.i.d.) Gaussian noise.
In addition, $A \in \Re^{n_x \times n_x}$, $B \in \Re^{n_x\times n_u}$, and $D \in \Re^{n_x \times n_w}$ are constant system matrices. 
The noise $w_k$ has the following properties,
\begin{equation}
	\mathbb{E}\left[w_k\right] = 0, \qquad \qquad
	\mathbb{E}\left[w_{k_1}{w_{k_2}^{\top}}\right] =
	I_{n_w}\delta_{k_1,k_2},
	\label{eq:GaussianNoise}
\end{equation}
where $\delta_{k_1,k_2}$ is the Kronecker delta function.
In addition, we have the following condition
\begin{equation}
	\mathbb{E}\left[x_{k_1}w_{k_2}^{\top}\right]=0,\qquad0 \leq k_1 \leq k_2.
\end{equation}
which stems from causality considerations.

Finally, we assume perfect state information and ignore the measurement noise, and thus
\begin{equation}\label{eq:x0}
	\Expectation[x_0] = \mu_0 = x_0,\qquad
	\Sigma_0 = 0.
\end{equation}

In the stochastic equation description (\ref{eq:SystemDynamics}) the system matrices $A, B$, and $D$ are assumed to be known.
The case when these matrices are not exactly known or depend on some random parameters is
out of the scope of this current paper, but is an important problem in its own right.
Probabilistic MPC approaches that can also handle parametric uncertainties in the system matrices have been proposed in the literature;
see, for example~\cite{kouvaritakis2016book} and the references therein.

It will be assumed that
the state and control inputs in (\ref{eq:SystemDynamics}) are subject to the constraints
\begin{align}\label{eq:state_and_input_const}
	x_k \in \mathcal{X},\quad u_k \in \mathcal{U},
\end{align}
for all $k \geq 0$, where $\mathcal{X} \subseteq \Re^{n_x}$ and $\mathcal{U}\subseteq \Re^{n_u}$ are convex sets containing the origin.
Throughout this work, we assume that the sets $\mathcal{X}$ and $\mathcal{U}$ are convex polytopes, represented as the intersection of a finite number of linear inequality constraints as follows
\begin{align}
	\mathcal{X} &\triangleq \bigcap_{i = 0}^{N_s-1} \left\{x: \alpha_{x,i}^\top x \leq \beta_{x,i}\right\},\label{eq:Xdefinition}\\
	\mathcal{U} &\triangleq \bigcap_{j = 0}^{N_c-1} \left\{u: \alpha_{u,j}^\top u \leq \beta_{u,j}\right\},\label{eq:Udefinition}
\end{align}
where $\alpha_{x,i}\in\Re^{n_x}$ and $\alpha_{u,j}\in\Re^{n_u}$ are constant vectors, and $\beta_{x,i}$ and $\beta_{u,j}$ are constant scalars.
In~(\ref{eq:Xdefinition}) and (\ref{eq:Udefinition}), $N_s$ and $N_c$ denote the number of state and control constraints defining the polytopes, respectively.
Notice that, since the disturbance $w_k$ in~(\ref{eq:SystemDynamics}) is possibly unbounded, the state may be unbounded as well.
Thus, we formulate the state constraints $x_k \in \mathcal{X}$ probabilistically, in terms of chance constraints
\begin{align}\label{eq:origCC}
	\Pr(x_k \in \mathcal{X}) \geq 1 - \epsilon_x,
\end{align}
where $\epsilon_x \ge 0$ is the maximum probability of constraint violation. 
In practice, typically, $\epsilon_x \ll 1$.
In this work we restrict the range of $\epsilon_x$ to the interval $\epsilon_x \in [0, 0.5)$, which will allow us to provide an alternative, 
deterministic and convex formulation of the chance constraint (\ref{eq:origCC}).
Using Boole's inequality~\cite{prekopa1988boole}, the constraints (\ref{eq:Xdefinition}) and (\ref{eq:origCC}) are satisfied, assuming the inequality
\begin{align}
	\Pr\left(\alpha_{x,i}^\top x_k \leq \beta_{x,i} \right) &\geq 1 - p_{x,i},
\end{align}
holds for all $i=0,\ldots,N_s-1$, where $p_{x,i}$ are such that
\begin{align}\label{eq:p_xj}
	\sum_{i=0}^{N_s-1} p_{x,i}\leq \epsilon_x,
\end{align}
where $p_{x,i} \in [0, 0.5)$ for all $i = 0, \ldots, N_s-1$~\cite{blackmore2009convex}. 
Similarly, by replacing the second inclusion in~(\ref{eq:state_and_input_const}) with the chance constraint
\begin{align}
	\Pr(u_k \in \mathcal{U}) \geq 1 - \epsilon_u,
\end{align}
where $\epsilon_u \in [0, 0.5)$, and along with~(\ref{eq:Udefinition}), we impose the following conditions
\begin{subequations} \label{eq:p_us}
\begin{align}
	\Pr\left(\alpha_{u,j}^\top u_k \leq \beta_{u,j} \right) &\geq 1 -  p_{u,j},   \label{eq:p_usA}\\
	\sum_{j=0}^{N_c-1} p_{u,j} &\leq \epsilon_u,\label{eq:p_usB}
\end{align}
\end{subequations}	
where $p_{u,j} \in [0, 0.5)$ for all $j = 0, \ldots, N_c-1$.

Our objective is to find a control sequence $\{u_k\}_{k=0}^\infty$ to minimize
\begin{equation}
J_\infty(x_0; u_0,u_1,\ldots) = \mathbb{E}\left[\sum_{k=0}^{\infty}x_k^\top Q x_k + u_k^\top R u_k\right], 
\end{equation}
where $Q \succeq 0$ and $R \succ 0$.

In summary, the aim of this work is to design a control sequence $\{u_k\}_{k = 0}^\infty$ that solves the following infinite horizon optimal control problem with state and control expectation-dependent quadratic cost, subject to chance constraints on the state and the input as follows.
\begin{tcolorbox}[width = \columnwidth, colback = white, arc = 0pt]\vspace{-15pt}
	\begin{subequations}\label{prob:InfiniteHorizonOptimalControlProblem}
		\begin{align}
		\min_{u_0,u_1,\ldots}	J_\infty(x_0; u_0,u_1,\ldots) &= \mathbb{E}\left[\sum_{k=0}^{\infty}x_k^\top Q x_k + u_k^\top R u_k\right], \\
		\textrm{subject to }&\nonumber \\
		& x_{k+1} = Ax_k + Bu_k + D w_k,\quad x_0 \sim \mathcal{N}(\mu_0,\Sigma_0), \\
		& \Pr\left(\alpha_{x,i}^\top x_k \leq \beta_{x,i} \right) \geq 1 - p_{x,i},\quad i = 0,\ldots,N_s-1,  \label{eq:InfinitecontrolSC} \\
		& \Pr\left(\alpha_{u,j}^\top u_k \leq \beta_{u,j} \right) \geq 1 - p_{u,j},\quad j = 0,\ldots,N_c-1,\label{eq:InfinitecontrolCC}
		\end{align}
	\end{subequations}
	where $Q \succeq 0$ and $R \succ 0$.
\end{tcolorbox}

\begin{Remark}
In the problem formulation (\ref{prob:InfiniteHorizonOptimalControlProblem}) we have probabilistic constraints for both the state and the control input.
In practice one would prefer, of course, to impose hard constraints in terms of the input.
Although hard input constraints can be incorporated in covariance steering problems~\cite{OkaTsi:cdc19},
enforcing input hard constraints for SMPC problems and guaranteeing stability and feasibility is not possible if the disturbance is unbounded and the system is not Schur stable~\cite{hokayem2009stochastic} or Lyapunov stable~\cite{hokayem2010stable}.
This is not a major issue for most engineering applications however, since almost sure satisfaction of the constraint (\ref{eq:InfinitecontrolCC}) can be ensured by choosing a very small value of $p_{u,j}$.
For other alternatives, see~\cite{KordaCiglerOlde2011,farina2015approach}. 
For more details, on the difference between hard and chance input constraints in SMPC problems
we refer the reader to the discussion in~\cite{chatterjee2012mean,paulson2017stochastic}.

\end{Remark}

\subsection{SMPC Formulation}

The SMPC aims to approximately solve the infinite-horizon optimal control problem~(\ref{prob:InfiniteHorizonOptimalControlProblem}) by solving, at each time step $k$, the following \emph{finite} horizon optimal control problem, instead.
\begin{tcolorbox}[width = \columnwidth, colback = white, arc = 0pt]\vspace{-15pt}
	\begin{subequations}\label{prob:FiniteHorizonOptimalControlProblem}
		\begin{align}
		\min_{u_{k|k},\ldots,u_{k+N-1|k}} &J_N(\mu_k,\Sigma_k;u_{k|k},\ldots,u_{k+N-1|k}) = \mathbb{E}_k\left[\sum_{t=k}^{k+N-1}x_{t|k}^\top Q x_{t|k} + u_{t|k}^\top R u_{t|k}\right] + J_f(x_{k+N|k}),\label{eq:originalObjFunc} \\
		\textrm{subject to }& \nonumber\\
		& x_{t+1|k} = A x_{t|k} + Bu_{t|k} + D w_t,\quad x_{k|k} = x_k \sim \mathcal{N}(\mu_k,\Sigma_k),\label{eq:MPCDynamics}\\
		& \Pr_k\left(\alpha_{x,i}^\top x_{t|k} \leq \beta_{x,i} \right) \geq 1 - p_{x,i},\quad i = 0,\ldots, N_s-1, \label{eq:MPCCC}\\
		& \Pr_k\left(\alpha_{u,j}^\top u_{t|k} \leq \beta_{u,j} \right) \geq 1 - p_{u,j},\quad j = 0,\ldots, N_c-1, \label{eq:MPCcontrolC}
		\end{align}
	\end{subequations}
	where $t = k, \ldots, k+N-1$ and $N$ is the optimization horizon.
\end{tcolorbox}

The notation $x_{t|k}$ denotes the state at time step $t$ predicted at time step $k \ge 0$ where $t \ge k$.
The variables $\mu_k$ and $\Sigma_k$ in (\ref{eq:MPCDynamics}) are the mean and the covariance of the state $x_k$, and are assumed to be given at step $k$.

We denote the optimal solution to~(\ref{prob:FiniteHorizonOptimalControlProblem}) as $\{u^\ast_{k|k},\ \ldots,\ u^\ast_{k+N-1|k}\}$.
At time step $k$, we apply $u^\ast_{k|k}$ to the system~(\ref{eq:SystemDynamics}), i.e., $	u_k = u_{k|k}^\ast$.
The function $J_f(\cdot): \Re^{n_x} \mapsto \Re$ is a terminal cost that needs to be designed properly to ensure stability~\cite{mayne2000constrained,farina2016stochastic,mesbah2016stochastic}.
In this work, we show that optimal covariance steering theory helps us choose an appropriate expression for $J_f(\cdot)$ to solve Problem~(\ref{prob:FiniteHorizonOptimalControlProblem})  efficiently and robustly.

\section{Optimal Covariance Steering} \label{sec:Preliminaries}

In this section, we introduce the basic theory behind optimal covariance steering controller design under state and control chance constraints, which will be applied to solve the SMPC problem~(\ref{prob:FiniteHorizonOptimalControlProblem}).
In the discrete-time optimal covariance steering problem setup~\cite{okamoto2018Optimal}, we wish to steer the state distribution of system~(\ref{eq:SystemDynamics}) from an initial Gaussian distribution
\begin{align}\label{eq:x0_OCS}
	x_0 \sim \mathcal{N}(\mu_0, \Sigma_0),
\end{align}
to a prescribed Gaussian distribution at a given time step $N$, i.e., 
\begin{align}
	x_N = x_f \sim \mathcal{N}(\mu_f, \Sigma_f).
\end{align}
Specifically, given an initial distribution~(\ref{eq:x0_OCS}), we wish to solve the following optimization problem.

\begin{tcolorbox}[width = \columnwidth, colback = white, arc = 0pt]\vspace{-15pt}
	\begin{subequations}\label{prob:OptimalCovarianceSteering}
		\begin{align}
			\min_{u_0,\ldots,u_{N-1}} &J(u_0,\ldots,u_{N-1}) = \mathbb{E}_k\left[\sum_{k=0}^{N-1}x_k^\top Q x_k + u_k^\top R u_k\right],\label{eq:CSObjFunc} \\
			\textrm{subject to } &\nonumber \\
			& x_{k+1} = Ax_k + Bu_k + D w_k,\quad x_0 \sim \mathcal{N}(\mu_0, \Sigma_0)\label{eq:SystemDynamicsOCS} \\
			& \Pr\left(\alpha_{x,i}^\top x_k \leq \beta_{x,i} \right) \geq 1 - p_{x,i}, \quad i = 0,\ldots, N_s-1,\label{eq:CSCCx}\\
			& \Pr\left(\alpha_{u,j}^\top u_k \leq \beta_{u,j} \right) \geq 1 - p_{u,j}, \quad j = 0,\ldots, N_c-1,\label{eq:CSCCu}\\
			& x_N = x_f \sim\mathcal{N}(\mu_f,\Sigma_f),\label{eq:CSxN}
		\end{align}
	\end{subequations}
	for $k = 0, \ldots, N-1$, where we assume that $\Sigma_0 \succeq 0$ and  $\Sigma_f \succ 0$, and $w_k$, $p_{x,i}$, and $p_{u,j}$ as in~(\ref{eq:GaussianNoise}),~(\ref{eq:p_xj}), and~(\ref{eq:p_us}), respectively.
\end{tcolorbox}

In order to solve problem (\ref{prob:OptimalCovarianceSteering}) we make the following assumptions to develop the finite horizon optimal covariance steering theory.

\begin{Assumption}		\label{remark:CSAssumption_}
The pair $(A,B)$ in (\ref{eq:SystemDynamics})  is controllable.
\end{Assumption}

\begin{Assumption}	\label{ass:Horizon_}
The horizon $N \ge n_x$.  This assumption, along with Assumption~\ref{remark:CSAssumption_}, ensures  that
 $x_f$ is reachable from $x_0$ for any $x_f \in \Re^{n_x}$, provided that $w_k = 0$ for $k = 0,\ldots,N-1$ and no state and control constraints.
	This assumption implies that, given any $x_f \in \Re^{n_x}$ and $x_0 \in \Re^{n_x}$, there exists a sequence of control inputs $\{u_0,\ldots,u_{N-1}\}$ that steers $x_0$ to $x_f$ in the absence of disturbances or any constraints.
\end{Assumption}

In order to proceed, we first rewrite the system dynamics in a more convenient form.
Following~\cite{goldshtein2017finite}, it is straightforward to obtain the following equivalent form of the system dynamics~(\ref{eq:SystemDynamicsOCS}), as follows
\begin{equation}\label{eq:X=Ax0+BU+DW}
	X = \ScriptA x_{0} + \ScriptB U+\ScriptD W,
\end{equation}
where
\begin{align}\label{eq:XUW}
	X = \begin{bmatrix}
	x_{0}\\x_{1}\\ \vdots \\ x_{N}
	\end{bmatrix}\in \Re^{(N+1) n_x},
	\quad
	U = \begin{bmatrix}
	u_{0}\\u_{1}\\ \vdots \\ u_{N-1}
	\end{bmatrix}\in\Re^{Nn_u},
	\quad
	W = \begin{bmatrix}
	w_{0}\\w_{1}\\ \vdots \\ w_{N-1}
	\end{bmatrix}\in \Re^{Nn_w},
\end{align}
and
where $\ScriptA\in\Re^{(N+1)n_x\times n_x}$, $\ScriptB \in\Re^{(N+1)n_x\times Nn_u}$, and $\ScriptD \in\Re^{(N+1)n_x\times Nn_w}$.
The explicit expressions for these matrices can be found, e.g., in~\cite{goldshtein2017finite}.
The initial conditions and noise satisfy
\begin{equation} \label{eq:Ex0x0x0WWW}
	\mathbb{E}[x_0x_0^\top] = \Sigma_0 + \mu_0\mu_0^\top,\qquad
	\mathbb{E}[x_0W^\top] = 0, \qquad
	\mathbb{E}[WW^\top] = I_{Nn_w}.
\end{equation}
Introducing the matrices
	\begin{align*}
			E_k &= \left[0_{n_x,kn_x}, I_{n_x},0_{n_x,(N-k)n_x}\right]\in \Re^{n_x\times(N+1)n_x}, \quad k = 0,\ldots,N, \\
			F_k &= \left[0_{n_u,kn_u}, I_{n_u},0_{n_u,(N-k-1)n_u}\right]\in \Re^{n_u\times Nn_u}, \quad k = 0,\ldots,N-1,
	\end{align*}
we have that the state and control at time step $k$ can be expressed in terms of $X$ and $U$ via $x_k = E_k X$ and $u_k = F_k U$.	

\begin{Proposition}
	Given~(\ref{eq:x0_OCS}), one can derive the following equivalent form of Problem~(\ref{prob:OptimalCovarianceSteering}) using~(\ref{eq:X=Ax0+BU+DW}), (\ref{eq:XUW}), and (\ref{eq:Ex0x0x0WWW}).
	
	\begin{tcolorbox}[width = \columnwidth, colback = white, arc = 0pt]\vspace{-15pt}
		\begin{subequations}\label{prob:OptimalCovarianceSteeringConverted}
			\begin{align}
			\min_{U}\ &J(U) =  \mathbb{E}\left[X^\top \bar{Q} X + U^\top \bar{R} U \right],\label{eq:CSObjFuncConverted}\\
			\textrm{subject to }\nonumber \\
			& X = \ScriptA x_{0} + \ScriptB U + \ScriptD W,\quad  x_0 \sim \mathcal{N}(\mu_0, \Sigma_0),\\
			&\Pr\left(\alpha_{x,i}^\top E_k X \leq \beta_{x,i} \right) \geq 1 - p_{x,i},\quad i = 0,\ldots, N_s - 1\label{eq:CSCCX}\\
			& \Pr\left(\alpha_{u,j}^\top F_k U \leq \beta_{u,j} \right) \geq 1 - p_{u,j},\quad j = 0,\ldots,N_c - 1 \label{eq:CSCCU}\\
			&\mu_f = E_N\mathbb{E}[X],\\
			&\Sigma_f = E_N\left(\mathbb{E}[XX^\top] - \mathbb{E}[X]\mathbb{E}[X]^\top\right) E_N^\top,\label{eq:CSTermCovConst}
			\end{align}
		\end{subequations}
		for $k = 0,\ldots,N-1$, where 
		\begin{align*}
		\bar{Q} &= \texttt{blkdiag}(Q,\ldots,Q,0) \in \Re^{(N+1)n_x\times(N+1)n_x},\\
		\bar{R} &= \texttt{blkdiag}(R,\ldots,R) \in \Re^{Nn_u\times Nn_u}.
		\end{align*}
	\end{tcolorbox}
\end{Proposition}

\begin{proof}
	The proof is straightforward from the discussion in~\cite{okamoto2018Optimal}.
	Note also that, because $Q\succeq0$ and $R\succ0$, it follows that $\bar{Q}\succeq0$ and $\bar{R}\succ0$.
\end{proof}

The following theorem shows that Problem~(\ref{prob:OptimalCovarianceSteeringConverted}) can be relaxed to a convex programming problem.

\begin{Theorem}\label{theorem:OCS}
	Given~(\ref{eq:x0}),~(\ref{eq:X=Ax0+BU+DW}),~(\ref{eq:XUW}),~(\ref{eq:Ex0x0x0WWW}), and the relaxation of~(\ref{eq:CSTermCovConst}) to the inequality
	\begin{align}\label{eq:CSTermCovConstRelaxed}
		\Sigma_f \succeq E_N\left(\mathbb{E}[XX^\top] - \mathbb{E}[X]\mathbb{E}[X]^\top\right) E_N^\top,
	\end{align}
	along with the control law
	\begin{align} \label{eq:OCSController}
		u_k &= v_k + K_k y_k,
	\end{align}
	where $v_k \in \Re^{n_u}$, $K_k \in \Re^{n_u \times n_x}$, and $y_k\in \Re^{n_x}$ from
	\begin{subequations}\label{eq:yDynamics}
		\begin{align}
			y_{k+1} &= A y_k + D w_k,\\
			y_0 &= x_0 - \mu_0, \label{eq:y0}
		\end{align}
	\end{subequations}
	reformulates Problem~(\ref{prob:OptimalCovarianceSteeringConverted}) as the following convex programming problem.
	\begin{tcolorbox}[width = \columnwidth, colback = white, arc = 0pt]\vspace{-15pt}
		\begin{subequations}\label{prob:OCSInputChanceConstrained}
			\begin{align}
			\min_{V,K}\ J(V, K) &= \trace\left[\left((I+\ScriptB K)^\top\bar{Q}(I+\ScriptB K)+K^\top\bar{R}K\right)\Sigma_y\right] \nonumber \\
			& \qquad + (\ScriptA\mu_0 + \ScriptB V)^\top \bar{Q} (\ScriptA\mu_0 + \ScriptB V) + V^\top \bar{R} V.\\
			\textrm{subject to}\nonumber\\
			&\alpha_{x,i}^\top E_k (\ScriptA\mu_0 + \ScriptB V) - \beta_{x,i} + \| \Sigma_y^{1/2}(I+\ScriptB K)^\top E_k^\top \alpha_{x,i}\|\Phi^{-1}(1-p_{x,i}) \leq 0,\label{eq:PrX}\\
			&\alpha_{u,j}^\top F_k V  - \beta_{u,j} + \| \Sigma_y^{1/2} K^\top F_k^\top \alpha_{u,j}\|\Phi^{-1}(1-p_{u,j})\leq 0,\label{eq:PrU}\\
			&\mu_f = E_N(\ScriptA\mu_0 + \ScriptB V),\\
			&\Sigma_f \succeq E_N(I+\ScriptB K)\Sigma_y(I+\ScriptB K)^\top E_N^\top,\label{eq:Sigmaf>ENIplusBKSigmaYIplusBK}
			\end{align}
		\end{subequations}
		for $i = 0,\ldots,N_s - 1$ and $j = 0,\ldots,N_c - 1$, where
		\begin{align*}
			\Sigma_y &= \ScriptA\Sigma_0\ScriptA^\top + \ScriptD\ScriptD^\top,
		\end{align*}
		and
		\begin{align*}
			V = \begin{bmatrix}
			v_0\\\vdots\\v_{N-1}
			\end{bmatrix},\qquad
			K = \begin{bmatrix}
			K_0 & & &  & 0 \\
			& K_1 & & & 0\\
			& & \ddots & & 0\\
			& & & K_{N-1}& 0
			\end{bmatrix},
		\end{align*}
		and where $\Phi^{-1}(\cdot)$ is the inverse function of the cumulative distribution function of the standard normal distribution.
	\end{tcolorbox}
\end{Theorem}

\begin{proof}
	All steps to convert Problem~(\ref{prob:OptimalCovarianceSteeringConverted}) to Problem~(\ref{prob:OCSInputChanceConstrained}) have already been discussed in our previous work~\cite{okamoto2018Optimal,okamoto2018Optimalb} except for the conversion of~(\ref{eq:CSCCU}) to~(\ref{eq:PrU}), which we describe below. 

To this end, notice that using the control law~(\ref{eq:OCSController}), the control sequence $U$ in~(\ref{eq:XUW}) is represented as
	\begin{align}
		U = V + K Y,
	\end{align}
	where $Y = \begin{bmatrix}y_0^\top & \cdots & y_N^\top \end{bmatrix}^\top \in \Re^{(N+1)n_x}$.
	It follows from~(\ref{eq:yDynamics}) that
	\begin{align}
		Y = \ScriptA y_0 + \ScriptD W,
	\end{align}
	and thus, using the facts that $\Expectation[y_0] = 0$, $\Expectation[y_0 y_0^\top] = \Sigma_0$, and $\Expectation[y_0 W^\top] = 0$, one obtains
	\begin{align}
		\Expectation[Y] = 0,\quad
		\Expectation[YY^\top] = \Sigma_y.
	\end{align}
	Therefore,
	\begin{align}
		\Expectation[U] = V,\quad
		\Expectation[\tilde{U}\tilde{U}^\top] =K\Sigma_y K^\top,
	\end{align}
	where $\tilde{U} = U-\Expectation[U]$.
	The inequality~(\ref{eq:CSCCU}) can be rewritten as
	\begin{align}\label{eq:Prau>b > 1minumsp_us}
		\Pr\left(\alpha_{u,j}^\top F_k (V + KY) \leq \beta_{u,j} \right) \geq 1-p_{u,j}.
	\end{align}
	Notice that $\alpha_{u,j}^\top F_k (V + KY)$ is a Gaussian distributed scalar random variable with mean $\alpha_{u,j}^\top F_kV$ and variance $\alpha_{u,j}^\top F_k K\Sigma_y K^\top F_k^\top \alpha_{u,j}$.
	Thus, inequality~(\ref{eq:Prau>b > 1minumsp_us}) becomes
	\begin{align}
		\Pr\left(\alpha_{u,j}^\top F_k (V + KY) \leq \beta_{u,j} \right) &=
		\frac{1}{\sqrt{2\pi\alpha_{u,j}^\top F_k K\Sigma_y K^\top F_k^\top \alpha_{u,j}}}
		\int_{-\infty}^{\beta_{u,j}}\exp\left(-\frac{(\xi - \alpha_{u,j}^\top F_kV)^2}{2\alpha_{u,j}^\top F_k K\Sigma_y K^\top F_k^\top \alpha_{u,j}}\right)\mathrm{d}\xi,\nonumber \\
		&=\Phi\left(\frac{\beta_{u,j} - \alpha_{u,j}^\top F_kV}{\sqrt{\alpha_{u,j}^\top F_k K\Sigma_y K^\top F_k^\top \alpha_{u,j}}}\right) \geq 1 - p_{u,j}.
	\end{align}
	Using the inverse function of $\Phi(\cdot)$, we obtain
	\begin{align}
	\alpha_{u,j}^\top F_k V - \beta_{u,j} + \sqrt{\alpha_{u,j}^\top F_k K\Sigma_y K^\top F_k^\top \alpha_{u,j}}\Phi^{-1}(1 - p_{u,j}) \leq 0,
	\end{align}
	which can be readily converted to~(\ref{eq:PrU}).
\end{proof}

Since $p_{x, i} \in [0, 0.5)$ and $p_{u, j} \in [0, 0.5)$, the constraints~(\ref{eq:PrX}) and~(\ref{eq:PrU}) are convex. 
Since Problem~(\ref{prob:OCSInputChanceConstrained}) is convex, one can efficiently solve the problem using a convex programming solver.
Specifically, because the terminal covariance constraint~(\ref{eq:Sigmaf>ENIplusBKSigmaYIplusBK}) can be converted to a linear matrix inequality (LMI),
\begin{align}
	\begin{bmatrix}
	\Sigma_f & E_N(I + \ScriptB K)\Sigma_y^{1/2} \\
	\Sigma_y^{1/2}(I + \ScriptB K)^\top E_N^\top & I
	\end{bmatrix} \succeq 0,
\end{align}
a semidefinite programming (SDP) solver such as Mosek~\cite{mosek} can be used.

\section{CS-SMPC Controller Design}    \label{sec:CS_SMPCDesign}

In the previous section, we introduced the optimal covariance steering controller.
We are now ready to discuss the main result of this paper, namely the CS-SMPC algorithm, followed by a proof of recursive feasibility and guaranteed stability.

\subsection{CS-SMPC Formulation}

In this section, we solve Problem~(\ref{prob:InfiniteHorizonOptimalControlProblem}) approximately by solving Problem~(\ref{prob:FiniteHorizonOptimalControlProblem}) at each time step in a receding horizon manner.
Specifically, at time step $k$, we wish to solve the following \emph{finite} horizon stochastic optimal control problem.
\begin{tcolorbox}[width = \columnwidth, colback = white, arc = 0pt]\vspace{-15pt}
	\begin{subequations}\label{prob:SMPC_Formulation}
		\begin{align}
			\min_{u_{k|k},u_{k+1|k},\ldots, u_{k+N-1|k}} &J_{N}(x_k; u_{k|k},u_{k+1|k},\ldots, u_{k+N-1|k}) =  \nonumber \\
			&\mathbb{E}_k\left[\sum_{t=k}^{k+N-1}x_{t|k}^\top Q x_{t|k} + u_{t|k}^\top R u_{t|k}\right]  + \mathbb{E}_k[x_{k+N|k}]^\top P_{\rm mean} \mathbb{E}_k[x_{k+N|k}], \label{eq:SMPC_Cost}\\
			\textrm{subject to }&\nonumber \\
			& x_{t+1|k} = Ax_{t|k} + Bu_{t|k} + D w_t,\quad x_{k|k} = x_k \sim \mathcal{N}(\mu_k,\Sigma_k),\\
			& \Pr_k\left(\alpha_{x,i}^\top x_{t|k} \leq \beta_{x,i} \right) \geq 1 - p_{x,i},\ i = 0,\ldots, N_s-1,\label{eq:SMPC_stateChanceConstraints} \\
			& \Pr_k\left(\alpha_{u,j}^\top u_{t|k} \leq \beta_{u,j} \right) \geq 1 - p_{u,j},\ j = 0,\ldots, N_c-1, \label{eq:SMPC_inputChanceConstraints} \\
			& \Expectation_k\left[x_{k+N|k}\right] \in \mathcal{X}^\mu_f,\label{eq:SMPC_meanTermConst}\\
			& \Expectation_k\left[(x_{k+N|k}-\Expectation[x_{k+N|k}])(x_{k+N|k}-\Expectation[x_{k+N|k}])^\top\right] \preceq \Sigma_f,\label{eq:SMPC_covTermConst}
		\end{align}
	\end{subequations}
	where $\mu_k\in\Re^{n_x}$, $\Sigma_k \in \Re^{n_x\times n_x}$, $P_{\rm mean} \in \Re^{n_x \times n_x}$, $\mathcal{X}^\mu_f \subset \Re^{n_x}$, and $\Sigma_f \in \Re^{n_x \times n_x}$ are given. 
\end{tcolorbox}

Problem~(\ref{prob:SMPC_Formulation}) is illustrated in Fig.~\ref{fig:SchematicCSSMPC}.
At each time step $k$, the \textit{predicted} system state and the  \textit{predicted} control have to satisfy the constraints.
In addition, at the end of the optimization horizon, the state mean has to be in a set $\mathcal{X}_f^\mu$, denoted by the yellow polytope, and the system covariance has to be smaller than $\Sigma_f$, denoted by the yellow ellipse in Fig.~\ref{fig:SchematicCSSMPC}.

Problem~(\ref{prob:SMPC_Formulation}) results from Problem (\ref{prob:FiniteHorizonOptimalControlProblem}) by setting
\begin{align}
	J_f(x) &= \mathbb{E}_k[x]^\top P_{\rm mean} \mathbb{E}_k[x],\label{eq:terminalCost}
\end{align}
along with the state terminal constraints~(\ref{eq:SMPC_meanTermConst}) and (\ref{eq:SMPC_covTermConst}).
Adding terminal constraints is a common methodology to ensure recursive feasibility and stability for MPC problems~\cite{mayne2000constrained,farina2016stochastic,mesbah2016stochastic}.
In this section, we show that, by properly designing the initial state mean and covariance pair along with the terminal parameters of Problem~(\ref{prob:SMPC_Formulation}), i.e.,  $\mu_k$, $\Sigma_k$, $\mathcal{X}^\mu_f$, $\Sigma_f$, and $P_{\rm mean}$, we can achieve recursive feasibility and guaranteed stability.

\begin{figure}
	\centering
	\includegraphics[width=0.7\columnwidth]{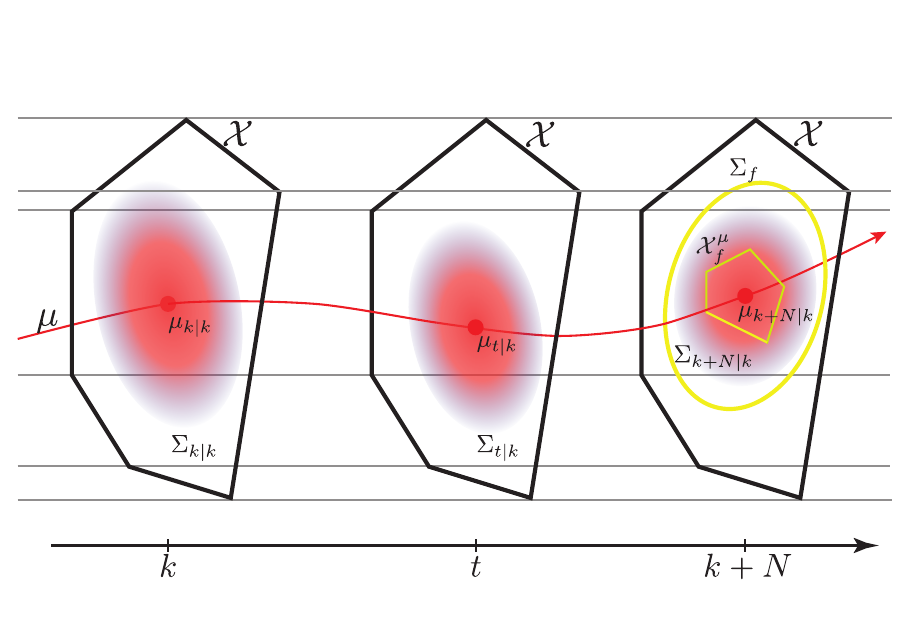}
	\caption{A schematic describing the proposed CS-SMPC approach. \label{fig:SchematicCSSMPC}}
\end{figure}

In order to solve problem~(\ref{prob:SMPC_Formulation}), similarly to Section~\ref{sec:Preliminaries}, we need the 
following additional assumption.

\begin{Assumption}	\label{remark:CSSMPCAssumption}
	All control channels are corrupted by noise, that is,
	$\mathcal{R}(B) \subseteq \mathcal{R}(D)$.
	This assumption, along with Assumption~\ref{remark:CSAssumption_}, 
	ensures that the pair $(A+BK,D)$ is controllable for any matrix $K$.
\end{Assumption}

We start with the following theorem that converts Problem~(\ref{prob:SMPC_Formulation}) to a more convenient form.

\begin{Theorem}
	Given $\mu_k$, $\Sigma_k$, $\mathcal{X}^\mu_f$, $\Sigma_f \succ 0$, and $P_{\rm mean} \succ 0$, and using the following control law
	\begin{subequations}    \label{eq:CSMPCControllerAll}
		\begin{align} \label{eq:CSMPCController}
			u_{t|k} &= v_{t|k} + K_{t|k} y_{t|k},
		\end{align}
		where $v_{t|k} \in \Re^{n_u}$, $K_{t|k} \in \Re^{n_u \times n_x}$, and $y_{t|k}\in \Re^{n_x}$ from
			\begin{align}
				y_{t+1|k} &= A y_{t|k} + D w_t,\\
				y_{k|k} &= x_{k|k} - \mu_{k|k},
			\end{align}
	\end{subequations}
	for $t=k,\ldots,k+N-1$, Problem~(\ref{prob:SMPC_Formulation}) can be cast as a convex programming problem as follows.
	\begin{tcolorbox}[width = \columnwidth, colback = white, arc = 0pt]\vspace{-15pt}
		\begin{subequations}\label{prob:SMPC_OCS}
			\begin{align}
			\min_{V, K}	J_{N}(\mu_{k}, \Sigma_{k}&; V, K) = \trace\left[\left((I+\ScriptB K)^\top\bar{Q}_{P,{\rm cov}}(I+\ScriptB K)+K^\top\bar{R}K\right)\Sigma_{y}\right] \nonumber \\
			& \qquad + (\ScriptA\mu_{k|k} + \ScriptB V)^\top \bar{Q}_{P,{\rm mean}} (\ScriptA\mu_{k|k} + \ScriptB V) + V^\top \bar{R} V.\\
			\textrm{subject to}&\nonumber \\
			&\alpha_{x,i}^\top E_{t-k}\left(\ScriptA \mu_{k|k} + \ScriptB V\right)+ \| \Sigma_{y}^{1/2}(I+\ScriptB K)^\top E_{t-k}^\top \alpha_{x,i}\|\Phi^{-1}(1-p_{x,i}) - \beta_{x,i} \leq 0,\label{eq:CS-SMPC_stateChanceConstraints} \\
			&\alpha_{u,j}^\top F_{t-k} V+ \|\Sigma_{y}^{1/2}K^\top F_{t-k}^\top\alpha_{u,j}\| \Phi^{-1}\left(1-p_{u,j}\right) - \beta_{u,j} \leq 0,\label{eq:CS-SMPC_inputChanceConstraints}\\
			&E_N\left(\ScriptA \mu_{k|k} + \ScriptB V\right) \in \mathcal{X}^\mu_f,\label{eq:CS-SMPC_meanTerminalConstraints}\\
			&\Sigma_f \succeq E_N(I+\ScriptB K)\Sigma_{y}(I+\ScriptB K)^\top E_N^\top,\label{eq:CS-SMPC_covTerminalConstraints}
			\end{align}
		\end{subequations}
		for all $i = 0,\ldots, N_s - 1$, $j = 0,\ldots, N_c - 1$, and $t = k, \ldots, k+N-1$,
		where
		\begin{subequations}\label{prob:SMPC_OCSe}
		\begin{align}
			& \mu_{k|k} = \mu_{k}, \qquad \Sigma_{k|k} = \Sigma_{k},\qquad \Sigma_y = \ScriptA\Sigma_{k|k}\ScriptA^\top + \ScriptD\ScriptD^\top,\\
			&V = \begin{bmatrix}
				v_{k|k}\\\vdots\\v_{k+N-1|k}
			\end{bmatrix},\qquad
			K = \begin{bmatrix}
			K_{k|k} & & &  & 0 \\
			& K_{k+1|k} & & & 0\\
			& & \ddots & & 0\\
			& & & K_{k+N-1|k}& 0
			\end{bmatrix},   \label{eq:Kmat}  \\
			&\bar{Q}_{P, {\rm mean}} =
			\begin{bmatrix}
			Q & & &  \\
			& \ddots  & \\
			& & Q & \\
			& & & P_{\rm mean}
			\end{bmatrix},\quad
			\bar{Q}_{P, {\rm cov}} =
			\begin{bmatrix}
			Q & & &  \\
			& \ddots  & \\
			& & Q & \\
			& & & 0
			\end{bmatrix},\quad
			\bar{R} =
			\begin{bmatrix}
			R & &   \\
			& \ddots &\\
			& & R
			\end{bmatrix}.
		\end{align}
		\end{subequations}
	\end{tcolorbox}
\end{Theorem}

\begin{proof}
	The proof follows directly from the discussion in Section~\ref{sec:Preliminaries} and thus is omitted.
\end{proof}

As discussed in Section~\ref{sec:Preliminaries}, Problem~(\ref{prob:SMPC_OCS}) can be efficiently solved using an SDP solver.
The remaining issue is whether Problem~(\ref{prob:SMPC_OCS}) is always feasible or not.
To this end, we need to impose proper initial and terminal conditions. 

\subsubsection{Initialization}    \label{sec:init}

In order to proceed we impose the following additional assumption.
\begin{Assumption}
	At time step $k = 0$, Problem~(\ref{prob:SMPC_OCS}) is feasible subject to the initial condition~(\ref{eq:x0}).
\end{Assumption}

Notice that, as we assume perfect state measurement, the value of $x_0$ is available.
Since we have unbounded additive noise, the state may become unbounded as well and
Problem~(\ref{prob:SMPC_OCS}) can become infeasible if we always set 
$\mu_k = x_k$ and $\Sigma_k = 0$.
In order to keep Problem~(\ref{prob:SMPC_OCS}) feasible for all time steps $k \geq 1$, several approaches have been proposed~\cite{farina2013probabilistic,paulson2017stochastic, hewing2018stochastic}.
In this work, we follow an initialization strategy similar to the one in~\cite{farina2013probabilistic}. 
Specifically, we set $\mu_k$ and $\Sigma_k$ as follows
\begin{equation} \label{initcond1}
\mu_k = x_k, \qquad \Sigma_k = 0
\end{equation}
and solve (\ref{prob:SMPC_OCS}).
If the problem is not feasible, we set
\begin{equation}  \label{initcond2}
 \mu_k = \mu^\ast_{k|k-1}, \qquad \Sigma_k = \Sigma^\ast_{k|k-1},
\end{equation}
which are the optimal \emph{predicted} mean and covariance from the previous time step.
Notice that  (\ref{initcond1}) is the result of closing the loop with the current measurement,
while the choice (\ref{initcond2}) does not use the most recent measurement and thus it can be treated as open-loop control.
Instead of 
\begin{equation} 
\Pr_k\left(\alpha_{x,i}^\top x_{t|k} \leq \beta_{x,i} \right) \geq 1 - p_{x,i},\quad i = 0,\ldots, N_s-1,~~ t = k+1,\ldots,k+N-1,
\end{equation} 
it enforces the constraint
\begin{equation} 
\Pr_k\left(\alpha_{x,i}^\top x_{k+1|k} \leq \beta_{x,i} \right) \geq 1 - p_{x,i},\quad i = 0,\ldots, N_s-1, 
\end{equation} 
for all $k=0,1,\ldots$
at the expense of suboptimal results in terms of cost function minimization~\cite{farina2016stochastic}.

Finally, in order to ensure feasibility we need to properly design the terminal constraints.
The idea of imposing the maximal terminal state covariance value in~(\ref{eq:SMPC_covTermConst}) has been 
proposed in~\cite{farina2013probabilistic}, in which the authors defined $\Sigma_f \succ 0$ as the steady-state solution of the following discrete-time Lyapunov equation
\begin{equation}\label{eq:FarinaTerminalCond}
	\Sigma_f = (A+BK_{\rm LQR})\Sigma_f (A+BK_{\rm LQR})^\top + DD^\top,
\end{equation}
where $K_{\rm LQR}$ is the infinite-horizon LQR gain.
Since $K_{\rm LQR}$ is determined from the $Q$ and $R$ matrices in~(\ref{eq:SMPC_Cost}), this approach results in $\Sigma_f$ being 
an implicit function of $Q$ and $R$ and not of the terminal penalty.
Thus, one needs to compromise performance in order to guarantee stability.
Alternatively, one has to ignore the stability of the controlled system in order to achieve the desired performance. 
The approach we propose in this work utilizes the results from covariance assignment theory to allow 
$\Sigma_f$ to be chosen independently from the $Q$ and $R$ matrices, as long as $\Sigma_f$ is \emph{assignable} at the end of the horizon.
Thus, one can design the $Q$ and $R$ matrices based on the desired system behavior, while ensuring guaranteed stability of the controlled system.

\subsubsection{Covariance Assignment Theory}

Before detailing how to choose $\mathcal{X}^\mu_f$, $\Sigma_f$, and $P_{\rm mean}$ in (\ref{prob:SMPC_Formulation}), we introduce the following results from 
covariance assignment theory~\cite{collins1987theory,grigoriadis1997minimum}.
Additional details can be found in~\cite{Skelton:book98}. 

\begin{Definition}[Assignable Covariance]
	The state covariance $\Sigma\succ 0$ is \emph{assignable} to the closed-loop system
	\begin{equation} \label{eq:closedLoopSystem}
		x_{k+1} = (A+B\tilde{K})x_k + D w_k,
	\end{equation}
	if $\Sigma$ satisfies
	\begin{equation}  \label{eq:closedLoopSteadyState}
		\Sigma = (A+B\tilde{K})\Sigma(A+B\tilde{K})^\top + DD^\top,
	\end{equation}
	where $\tilde{K} \in \Re^{n_u \times n_x}$ is a state-feedback gain. 	
\end{Definition}

It follows from this definition that the matrix $\Sigma_f$ in~(\ref{eq:FarinaTerminalCond}) is an assignable covariance with 
corresponding state-feedback gain $K_{\rm LQR}$.
Note also that since $\Sigma\succ 0$ it follows from (\ref{eq:closedLoopSteadyState}) and the fact that the pair $(A+B\tilde{K},D)$ is 
controllable (see Assumption~\ref{remark:CSSMPCAssumption})
that the matrix $A+B\tilde{K}$ is Hurwitz and hence  $\tilde{K}$ is stabilizing.

The assignable set of matrices $\Sigma$ and the corresponding stabilizing gain matrices $\tilde{K}$ can be computed as follows.

\begin{Proposition}[\cite{collins1987theory}]\label{theorem:AssignableSigma}
	The set of assignable state covariances $\Sigma$ can be parameterized by the following set of LMIs
	\begin{subequations}\label{eq:AssignableConds}
		\begin{align}
		(I - BB^+)(\Sigma - A\Sigma A^\top &- DD^\top)(I - BB^+) = 0,\\
		\Sigma &\succ 0,\\
		\Sigma &\succeq DD^\top,
		\end{align}
	\end{subequations}
	where $B^+$ denotes the Moore-Penrose pseudoinverse of $B$.
\end{Proposition}
\begin{Proposition}[\cite{collins1987theory}]\label{theorem:Kdesign}
	Let $\Sigma \succ 0$ be an assignable covariance matrix. 
	Then all (stabilizing) state-feedback gains $\tilde{K}$ that satisfy~(\ref{eq:closedLoopSteadyState}) are parametrized by
	\begin{align}\label{eq:KtildeFormula}
		\tilde{K} = B^+\left((\Sigma - DD^\top)^{1/2}G_1
		\begin{bmatrix}
			I_r & 0 \\ 0 & T
		\end{bmatrix}
		G_2^\top S^{-1} - A\right)
		+ (I_{n_u} - B^+B)Z,
	\end{align}
	where $T$ is an arbitrary orthogonal matrix, $SS^\top = \Sigma$, $Z \in \Re^{n_u\times n_x}$ is an arbitrary matrix, and $G_1$ and $G_2$ are defined from the singular-value decompositions
	\begin{subequations}\label{eq:ComputeSVD}
		\begin{align}
			(I - BB^+)(\Sigma - DD^\top)^{1/2} &= L \Lambda G_1^\top,\\
			(I - BB^+)AS &= L \Lambda G_2^\top,
		\end{align}
	\end{subequations}
	where $L$, $G_1$, and $G_2$ are orthogonal matrices, and $\Lambda = \mathtt{diag}(\sigma_1,\ldots,\sigma_{r},0,\ldots,0)$ with $\sigma_1 \geq \sigma_2 \geq\ldots\geq \sigma_r>0$.
\end{Proposition}

\begin{Remark}
	It is worth noticing~\cite{collins1987theory} that, if $A$ is nonsingular and $B$ is full column rank, then the rank $r$ in~(\ref{eq:ComputeSVD}) is  $r = n_x - n_u$. 
	In addition, $I_{n_u} - B^+B = 0$, and thus the second term in~(\ref{eq:KtildeFormula}) vanishes.
\end{Remark}

\subsubsection{Terminal Cost and Terminal Constraints Design}

We are now ready to prove that, by properly designing the terminal cost and terminal constraints ($\mathcal{X}^\mu_f$, $\Sigma_f$, and $P_{\rm mean}$), one can guarantee that the Problem~(\ref{prob:SMPC_OCS}) with initial condition (\ref{initcond2}) is feasible and results in feasible control commands.
In addition, we show that the stability of the closed-loop system with the proposed CS-SMPC algorithm in (\ref{prob:SMPC_OCS}) is guaranteed.
Let us denote the optimal cost of Problem~(\ref{prob:SMPC_OCS}) at time step $k$ by $J_N^\ast(x_{k|k})$ and the associated predicted
optimal control sequence by
\begin{align}\label{eq:CS-SMPC_ControlSequence}
	\{u_{k|k}^\ast,\ldots,u_{k+N-1|k}^\ast \} = \{v_{k|k}^\ast + K_{k|k}^\ast y_{k|k},\ldots,v_{k+N-1|k}^\ast + K_{k+N-1|k}^\ast y_{k+N-1|k}\},
\end{align}
which generates the corresponding predicted optimal state sequence~$\{x_{k|k}^\ast,x_{k+1|k}^\ast\ldots,x_{k+N|k}^\ast \}$ with $x_{k|k}^\ast = x_k$.
Since we are dealing with systems with additive uncertainty, it is difficult to design a control law that ensures the mean square stability of the state~\cite{cannon2009probabilisticConstrained}.
Instead, and similarly to~\cite{cannon2011stochastic}, we show that the average stage cost is bounded from above.

\begin{Theorem}\label{theorem:RecursiveFeasibility}
	Suppose that $\Sigma_f$ is assignable and satisfies~(\ref{eq:AssignableConds}),  $\mu_f$ satisfies $\mu_f \in \mathcal{X}^\mu_f$, where the set $\mathcal{X}^\mu_f \subset \Re^{n_x}$ is a positively invariant set~\cite{borrelli2017predictive} such that, for any $ \mu \in \mathcal{X}^\mu_f$,
	\begin{subequations}
		\begin{align}
			&(A+B\tilde{K})\mu \in \mathcal{X}^\mu_f,\label{eq:(AplusBK)muinX_f}\\
			&\alpha_{x,i}^\top \mu + \| \Sigma_f^{1/2} \alpha_{x,i}\|\Phi^{-1}(1-p_{x,i}) - \beta_{x,i} \leq 0,\quad i = 0,\ldots, N_s-1 \label{eq:RecFeasX}\\
			&\alpha_{u,j}^\top \tilde{K} \mu + \|\Sigma_f^{1/2}\tilde{K}^\top\alpha_{u,j}\| \Phi^{-1}\left(1-p_{u,j}\right) - \beta_{u,j} \leq 0,\quad j = 0,\ldots, N_c-1, \label{eq:RecFeasU}
		\end{align}
	\end{subequations}
	where $\tilde{K}$ is derived from~(\ref{eq:KtildeFormula}), and $P_{\rm mean}$ is the solution of the following discrete-time Lyapunov equation
	\begin{align}
		(A  + B\tilde{K})^\top P_{\rm mean} (A  + B\tilde{K}) - P_{\rm mean} + Q + \tilde{K}^\top R \tilde{K} = 0. \label{eq:meanCond}
	\end{align}
	Then, the solution of Problem~(\ref{prob:SMPC_OCS}) ensures recursive feasibility and stability.
	Namely, the following two properties hold:
	\begin{itemize}
	
		\item[a)] 
		If Problem~(\ref{prob:SMPC_OCS}) is feasible 
		at time step $k$, i.e., if the control sequence~(\ref{eq:CS-SMPC_ControlSequence}) satisfies~(\ref{eq:CS-SMPC_stateChanceConstraints}), (\ref{eq:CS-SMPC_inputChanceConstraints}), (\ref{eq:CS-SMPC_meanTerminalConstraints}), and~(\ref{eq:CS-SMPC_covTerminalConstraints}), then	Problem~(\ref{prob:SMPC_OCS}) is feasible for all $k+n$, where $n\geq 1$.
		
		\item[b)] 
		The average stage cost is bounded from above. Specifically,  there exists a positive $\ell_{\max} > 0$ such that 
		\begin{align}\label{eq:averageExCostIsBounded}
			\lim_{n\rightarrow\infty}\frac{1}{n}\sum_{t=0}^{n-1} \Expectation_k\left[x^{\ast\top}_{k+t|k} Q x^\ast_{k+t|k} + u_{k+t|k}^{\ast\top} R u_{k+t|k}^\ast\right] \leq \ell_{\max}.
		\end{align}
		
	\end{itemize}
\end{Theorem}

\begin{proof}
	In order to simplify notation, henceforth we will rewrite the cost function $J_{N}$  in~(\ref{eq:SMPC_Cost}) as
	\begin{align}
	J_{N}(x_{k|k}; u_{k|k},\ldots, u_{k+N-1|k}) = \sum_{t=k}^{k+N-1}\ell(x_{t|k},u_{t|k}) + J_f(x_{k+N|k}),
	\end{align}
	where
	\begin{align}
	\ell(x_{t|k},u_{t|k}) &= \Expectation_k\left[x_{t|k}^\top Q x_{t|k} + u_{t|k}^\top R u_{t|k}\right],\label{eq:runningCost}
	\end{align}
	and $J_f(\cdot)$ is as in~(\ref{eq:terminalCost}).
		
	In order to prove recursive feasibility, it is sufficient to show that, given that Problem~(\ref{prob:SMPC_OCS}) is feasible at time step $k$, it is feasible at time step $k+1$.
	Specifically, we show that, given Problem~(\ref{prob:SMPC_OCS}) is feasible at time step $k$, there exists at least one control sequence solving Problem~(\ref{prob:SMPC_OCS}) at time step $k+1$ with initial mean and covariance set to $\mu^\ast_{k+1|k}$ and $\Sigma^\ast_{k+1|k}$.
	To this end, we consider Problem~(\ref{prob:SMPC_OCS})
	with the following control sequence of length $N$
	\begin{align}\label{eq:recursiveFeasiblityTemporalControlSequence}
		\bm{u} = \{v^\ast_{k+1|k} + K_{k+1|k}^\ast y_{k+1|k},\ldots,v^\ast_{k+N-1|k} + K_{k+N-1|k}^\ast y_{k+N-1|k}, \tilde{K}x_{k+N|k}^\ast \},
	\end{align}
	where the first $N-1$ elements are derived from the optimal control sequence at time step $k$ in~(\ref{eq:CS-SMPC_ControlSequence}), and the last step is a covariance assignment control with gain as in~(\ref{eq:KtildeFormula}).
	This control sequence steers the predicted state trajectory from $x^\ast_{k+1|k}$ to
	\begin{align}\label{eq:stateSequencekPlus1}
		\bm{x} = \{x^\ast_{k+1|k},\ldots,x^\ast_{k+N|k}, (A + B\tilde{K})x_{k+N|k}^\ast + D w_{k+N}\}.
	\end{align}
	Note that the control sequence~(\ref{eq:recursiveFeasiblityTemporalControlSequence}) can be separated to the mean control sequence
	\begin{align}\label{eq:CS-SMPC_meanControlSequence}
		\{v^\ast_{k+1|k},\ldots,v^\ast_{k+N-1|k}, \tilde{K}\mu_{k+N|k}^\ast \},
	\end{align}
	and the covariance steering sequence
	\begin{align}\label{eq:CS-SMPC_covControlSequence}
		\{K_{k+1|k}^\ast y_{k+1|k},\ldots, K_{k+N-1|k}^\ast y_{k+N-1|k}, \tilde{K}(x_{k+N|k}^\ast-\mu_{k+N|k}^\ast) \}.
	\end{align}
	Since in the predicted state sequence~(\ref{eq:stateSequencekPlus1}) the first $N-1$ components follow the same path as the predicted solution at time step $k$, 
	we only need to check the satisfaction of the constraints (\ref{eq:SMPC_stateChanceConstraints}) (\ref{eq:SMPC_inputChanceConstraints}) (\ref{eq:SMPC_meanTermConst}), and (\ref{eq:SMPC_covTermConst}) at the end of the horizon.
	
	We first show that the mean state at the end of the horizon satisfies the terminal mean constraint~(\ref{eq:SMPC_meanTermConst}).
	The first $N-1$ mean control subsequence in~(\ref{eq:CS-SMPC_meanControlSequence}) steers $\mu^\ast_{k+1|k}$ to $\mu^\ast_{k+N|k}$.
	Because of the fact that $\mu^\ast_{k+N|k} \in \mathcal{X}^\mu_f$, the last entry in~(\ref{eq:CS-SMPC_meanControlSequence}) steers the system mean to
	\begin{equation} \label{eq:recursiveFeasibilitymu}
		\mu_{k+N+1|k} = (A + B\tilde{K})\mu^\ast_{k+N|k} \in \mathcal{X}^\mu_f.
	\end{equation}
	Thus, the constraint~(\ref{eq:SMPC_meanTermConst}) is satisfied at the end of the horizon.
	
	Next, we show that the terminal covariance constraint~(\ref{eq:SMPC_covTermConst}) is satisfied at the end of the horizon.
	Note that the first $N-1$ covariance control subsequence in~(\ref{eq:CS-SMPC_covControlSequence}) steers $\Sigma^\ast_{k+1|k}$ to $\Sigma^\ast_{k+N|k}$.
	It follows from~(\ref{eq:closedLoopSystem}) that
	\begin{align}\label{eq:SigmakpNp1givenk}
		\Sigma_{k+N+1|k} &= (A+B\tilde{K})\Sigma^\ast_{k+N|k}(A+B\tilde{K})^\top + DD^\top.
	\end{align}
	In addition, since $\Sigma_f$ is designed to be assignable, it follows from~(\ref{eq:closedLoopSteadyState}) that
	\begin{align}\label{eq:SigmafSteadyState}
		\Sigma_f &= (A+B\tilde{K})\Sigma_f(A+B\tilde{K})^\top + DD^\top.
	\end{align}
	It then follows from~(\ref{eq:SigmakpNp1givenk}),~(\ref{eq:SigmafSteadyState}), and the fact that $\Sigma^\ast_{k+N|k} \preceq \Sigma_f$, 
	that
	\begin{align}\label{eq:SigmakpNp1givenkLessThanSigmaf}
		\Sigma_{k+N+1|k} \preceq \Sigma_f,
	\end{align}
	which indicates the satisfaction of the condition~(\ref{eq:SMPC_covTermConst}) at the end of the horizon.
	
	The remaining constraints needed to be satisfied are~(\ref{eq:SMPC_stateChanceConstraints}) and~(\ref{eq:SMPC_inputChanceConstraints}).
	Note that, because of~(\ref{eq:recursiveFeasibilitymu}),
	\begin{align}
		\alpha_{x,i}^\top \mu_{k+N+1|k} + \| \Sigma_f^{1/2} \alpha_{x,i}\|\Phi^{-1}(1-p_{x,i}) - \beta_{x,i} \leq 0,\quad i = 0,\ldots, N_s-1,
	\end{align}
	holds.
	In addition, because of (\ref{eq:p_xj}) with $\epsilon_x \in [0,0.5)$, it follows that $p_{x,i} \leq 0.5$, and thus, $\Phi^{-1}(1-p_{x,i}) \geq 0$ for $i = 0,\ldots,N_s - 1$.
	Therefore, along with~(\ref{eq:SigmakpNp1givenkLessThanSigmaf}),
	\begin{align}
		\alpha_{x,i}^\top \mu_{k+N+1|k} + \| \Sigma_{k+N+1|k}^{1/2} \alpha_{x,i}\|\Phi^{-1}(1-p_{x,i}) - \beta_{x,i} \leq 0,\quad i = 0,\ldots, N_s-1,
	\end{align}
	which means that~(\ref{eq:SMPC_stateChanceConstraints}) is satisfied at the end of the horizon.
	Following a similar discussion, we can show that~(\ref{eq:SMPC_inputChanceConstraints}) is also satisfied at the end of the horizon.
	Namely,
	\begin{align}
		\alpha_{u,j}^\top \tilde{K} \mu_{k+N+1|k} + \|\Sigma_{k+N+1|k}^{1/2}\tilde{K}^\top\alpha_{u,j}\| \Phi^{-1}\left(1-p_{u,j}\right) - \beta_{u,j} \leq 0,\quad j = 0,\ldots, N_c-1.
	\end{align}
	Thus, we have shown that, given that Problem~(\ref{prob:SMPC_OCS}) is feasible at time step $k$, the control sequence in~(\ref{eq:CS-SMPC_meanControlSequence}) leads to the satisfaction of all the constraints in Problem~(\ref{prob:SMPC_OCS}) and~(\ref{prob:SMPC_Formulation}).
	The remaining issue is to show that the proposed control policy 
	$u_{k+N|k} = v_{k+N|k} +  K_{k+N|k}y_{k+N|k}$ 
	reproduces the same control input as $u_{k+N|k} = \tilde{K} x_{k+N|k}^\ast$.
	This can be achieved by letting
	\begin{subequations}
		\begin{align}
		\tilde{K}\mu_{k+N|k}^\ast&= v_{k+N|k} + K_{k+N|k}A^N y_{k|k},\\
		\tilde{K}(x_{k+N|k}^\ast - \mu_{k+N|k}^\ast) &= K_{k+N|k}(y_{k+N|k} - A^N y_{k|k}),
		\end{align}
	\end{subequations}
where
	\begin{align}\label{eq:y_k+Ngivenk}
		y_{k+N|k} = A^N y_{k|k} + \begin{bmatrix}
		A^{N-1}D & \cdots & AD & D 
		\end{bmatrix} 
		\begin{bmatrix}
		w_k \\ \vdots \\ w_{k+N-1}
		\end{bmatrix}.
	\end{align}
If $y_{k+N|k} \neq A^N y_{k|k}$, by letting
	\begin{subequations}
		\begin{align}
			K_{k+N|k} &=  \tilde{K} (x_{k+N|k}^\ast - \mu_{k+N|k}^\ast)\frac{1}{\|y_{k+N|k} - A^N y_{k|k}\|^2}(y_{k+N|k} - A^N y_{k|k})^\top,\\
			v_{k+N|k} &= \tilde{K}\mu_{k+N|k}^\ast - K_{k+N|k} A^N y_{k|k},
		\end{align}
	\end{subequations}
	yields the desired result.
If, on the other hand, $y_{k+N|k} = A^N y_{k|k}$, it follows from (\ref{eq:y_k+Ngivenk}) that
\begin{align}
	y_{k+N|k} -  A^N y_{k|k}  =  \begin{bmatrix}
	A^{N-1}D & \cdots & AD & D 
	\end{bmatrix} 
	\begin{bmatrix}
	w_k \\ \vdots \\ w_{k+N-1}
	\end{bmatrix}
	= 0,
\end{align}
and hence,  
\begin{align}
	x_{k+N|k}^\ast & = A^N x_{k|k} + \begin{bmatrix}
	A^{N-1}B & \cdots & AB &  B
	\end{bmatrix} 
	\begin{bmatrix}
	u_{k|k}^\ast \\ \vdots \\ u_{k+N-1|k}^\ast
	\end{bmatrix}
	+ 	\begin{bmatrix}
	A^{N-1}D & \cdots & AD &  D 
	\end{bmatrix}  
	\begin{bmatrix}
	w_k \\ \vdots \\ w_{k+N-1}
	\end{bmatrix},\\
	&= A^N x_{k|k} + \begin{bmatrix}
	A^{N-1}B & \cdots & AB & B
	\end{bmatrix} 
	\begin{bmatrix}
	u_{k|k}^\ast \\ \vdots \\ u_{k+N-1|k}^\ast
	\end{bmatrix}.
\end{align}
Thus  $x_{k+N|k}^\ast$ can be computed deterministically from the control inputs.
In this case, we can choose	
	\begin{subequations}
		\begin{align}
			v_{k+N|k} &= \tilde{K}x_{k+N|k}^\ast,\\
			K_{k+N|k} &= 0.
		\end{align}
	\end{subequations}  
So far, we have shown the recursive feasibility of the closed-loop system with CS-SMPC.
	Next, we discuss the issue of stability.
	Note that the cost $J_N(x^\ast_{k+1|k}; \bm{u})$ can be represented as
	\begin{align} \label{eq:Jevol}
		J_N(x^\ast_{k+1|k}; \bm{u}) = &J_N^\ast(x_{k|k})-\ell(x_{k|k},u_{k|k}^\ast) + \ell(x^\ast_{k+N|k},\tilde{K}x_{k+N|k}^\ast)\nonumber \\
			&- J_f(x^\ast_{k+N|k}) + J_f((A + B\tilde{K})x^\ast_{k+N|k} +D w_{k+N}).
	\end{align}
	We first show that
	\begin{align} \label{eq:JN_Inequality}
		J_N^\ast(x^\ast_{k+1|k}) \leq J_N^\ast(x_{k|k})-\ell(x_{k|k},u_{k|k}^\ast) + \mathtt{tr}\left((Q + \tilde{K}^\top R\tilde{K})\Sigma_f\right).
	\end{align}
	It follows from~(\ref{eq:terminalCost}) that
	\begin{align}
		J_f(x^\ast_{k+N|k}) &= \mu^{\ast\top}_{k+N|k} P_{\rm mean} \mu^\ast_{k+N|k}.
	\end{align}
	In addition, and since the mean of the system state at the end of horizon is $(A + B\tilde{K})\mu^\ast_{k+N|k}$, the following holds
	\begin{align}
		J_f((A + B\tilde{K}) x^\ast_{k+N|k}  + D w_{k+N} ) = \mu^{\ast\top}_{k+N|k}(A  + B\tilde{K})^\top P_{\rm mean} (A  + B\tilde{K})\mu^\ast_{k+N|k}.
	\end{align}
	Furthermore, it follows from~(\ref{eq:runningCost}) that
	\begin{align}
		\ell(x^\ast_{k+N|k}, \tilde{K}x_{k+N|k}^\ast) &= \mu_{k+N|k}^{\ast \top} (Q + \tilde{K}^\top R \tilde{K}) \mu^\ast_{k+N|k} + \mathtt{tr}\left((Q + \tilde{K}^\top R \tilde{K})\Sigma_{k+N|k}^\ast\right).
	\end{align}
	Thus, using the conditions~(\ref{eq:meanCond}) and~(\ref{eq:closedLoopSteadyState}), it follows from~(\ref{eq:Jevol}) that
	\begin{align}
		&J_N(x^\ast_{k+1|k},\bm{u}) - J_N^\ast(x_{k|k}) + \ell(x_{k|k}, u_{k|k}^\ast)= \ell(x^\ast_{t+N|t}, \tilde{K}x_{k+N|k}^\ast) - J_f(x^\ast_{t+N|t}) + J_f((A + B\tilde{K}) x^\ast_{t+N|t} +D w_{t+N}),\nonumber \\
		&= \mu_{t+N|t}^{\ast\top}\left(Q + \tilde{K}^\top R \tilde{K} - P_{\rm mean} + (A  + B\tilde{K})^\top P_{\rm mean} (A  + B\tilde{K}) \right)\mu^\ast_{t+N|t} + \mathtt{tr}\left((Q + \tilde{K}^\top R\tilde{K})\Sigma^\ast_{t+N|t}\right),\nonumber\\
		&= \mathtt{tr}\left((Q + \tilde{K}^\top R\tilde{K})\Sigma^\ast_{t+N|t}\right) = \mathtt{tr}\left((Q + \tilde{K}^\top R\tilde{K})^{1/2}\Sigma^\ast_{t+N|t}(Q + \tilde{K}^\top R\tilde{K})^{1/2}\right),\nonumber \\
		&\leq \mathtt{tr}\left((Q + \tilde{K}^\top R\tilde{K})^{1/2}\Sigma_f(Q + \tilde{K}^\top R\tilde{K})^{1/2}\right) = \mathtt{tr}\left((Q + \tilde{K}^\top R\tilde{K})\Sigma_f\right).
	\end{align}
	Note also that since
	\begin{align}
		J_N^\ast(x^\ast_{k+1|k}) \leq J_N(x^\ast_{k+1|k},\bm{u}),
	\end{align}
	inequality~(\ref{eq:JN_Inequality}) holds.
	It then follows from~(\ref{eq:JN_Inequality}) that
	\begin{align*}
		\ell(x_{k|k},u_{k|k}^\ast) &\leq J_N^\ast(x_{k|k}) - J_N^\ast(x^\ast_{k+1|k}) + \mathtt{tr}\left((Q + \tilde{K}^\top R\tilde{K})\Sigma_f\right),\\
		\ell(x^\ast_{k+1|k},u_{k+1|k}^\ast) &\leq J_N^\ast(x^\ast_{k+1|k}) - J_N^\ast(x^\ast_{k+2|k}) + \mathtt{tr}\left((Q + \tilde{K}^\top R\tilde{K})\Sigma_f\right),\\
		&\vdots \nonumber\\
		\ell(x^\ast_{k+n-1|k},u_{k+n-1|k}^\ast) &\leq J_N^\ast(x^\ast_{k+n-1|k}) - J_N^\ast(x^\ast_{k+n|k}) + \mathtt{tr}\left((Q + \tilde{K}^\top R\tilde{K})\Sigma_f\right),
	\end{align*}
	and thus,
	\begin{align}
		\lim_{n\rightarrow\infty}\frac{1}{n}\sum_{t=0}^{n-1} \ell(x^\ast_{k+t|k},u_{k+t|k}^\ast)
		\leq
		\lim_{n\rightarrow\infty} \frac{1}{n} \left(J_N^\ast(x_{k|k}) - J_N^\ast(x^\ast_{k+n|k})\right) + \mathtt{tr}\left((Q + \tilde{K}^\top R\tilde{K})\Sigma_f\right).
	\end{align}
	Since $J_N^\ast(\cdot)$ has a finite lower bound, the right-hand-side of this inequality is bounded from above.
	Thus, there exists a positive value $\ell_{\max}$ such that
	\begin{align}
		\lim_{n\rightarrow\infty}\frac{1}{n}\sum_{k=0}^{n} \ell(x^\ast_{t+k|t},u_{t+k|t}^\ast) \leq \mathtt{tr}\left((Q + \tilde{K}^\top R\tilde{K})\Sigma_f\right) = \ell_{\max}.
	\end{align}
	which leads to~(\ref{eq:averageExCostIsBounded}).
\end{proof}

\begin{Remark}
	As indicated by (\ref{eq:KtildeFormula}), the gain matrix $\tilde{K}$ that satisfies~(\ref{eq:closedLoopSteadyState}) is not unique.
	In our numerical implementation of CS-SMPC, we have used
	\begin{align} \label{eq:Ktilde}
		\tilde{K} = B^+\left((\Sigma_f - DD^\top)^{1/2}G_1G_2^\top \Sigma_f^{-1/2} - A\right),
	\end{align}
	which can be derived by setting $T=I$ and $Z = 0$.
\end{Remark}

\begin{Remark}
	We choose $\mathcal{X}^\mu_f$ to be the maximal positively invariant set
	 for the mean system dynamics
	\begin{align}
		\mu_{k+1} = A\mu_k + Bv_{k},
	\end{align}
	subject to the inequalities
	\begin{subequations}
		\begin{align}
			&\alpha_{x,i}^\top \mu_k + \| \Sigma_f^{1/2} \alpha_{x,i}\|\Phi^{-1}(1-p_{x,i}) - \beta_{x,i} \leq 0,\quad i = 0,\ldots, N_s-1\\
			&\alpha_{u,j}^\top \tilde{K} \mu_k + \|\Sigma_f^{1/2}\tilde{K}^\top\alpha_{u,j}\| \Phi^{-1}\left(1-p_{u,j}\right) - \beta_{u,j} \leq 0,\quad j = 0,\ldots, N_c-1,
		\end{align}
	\end{subequations}
	 for all $k = 0,1,\ldots,N$.
	Such a set can be computed efficiently from the results in~\cite{borrelli2017predictive}.
\end{Remark}

\begin{Remark}
	Because the eigenvalues of $A+B\tilde{K}$ lie inside the unit ball and $Q+\tilde{K}R\tilde{K} \succeq 0$, it follows from~(\ref{eq:meanCond}) that $P_{\rm mean} \succeq 0$, and thus, the cost function~(\ref{eq:SMPC_Cost}) is convex.
\end{Remark}

\section{Numerical Simulations}   \label{sec:Numerical Simulation}

In this section we validate the proposed algorithm using two numerical examples.
In the first example, we illustrate the benefits of CS-SMPC using a problem with simple dynamics.
This example was chosen to compare some of the difficulties with implementing some of the existing SMPC methods and also demonstrate the computational benefits of the control parameterization (\ref{eq:CSMPCControllerAll}) used by CS-SMPC.
In the second example, we demonstrate that CS-SMPC can be applied to control an autonomous racing vehicle.
For all numerical calculations in this section we implemented in MATLAB using YALMIP~\cite{yalmip} along with the MPT3 toolbox~\cite{MPT3} to compute the maximal invariant sets and 
used MOSEK~\cite{mosek} to solve the relevant optimization problems.
All computations were performed on a computer with an Intel Xeon E5-1650 v4 @ 3.60GHz CPU processor with 64GB of RAM running Windows 10 OS.

\subsection{Illustrative Example with 2D Dynamics \label{subsec:NSSimpleExample}}

In this section, we demonstrate the benefit of CS-SMPC using a numerical example similar to the one used in~\cite{primbs2009stochastic}.
We set the system dynamics matrices in~(\ref{eq:SystemDynamics}) to be
\begin{equation}
	A = \begin{bmatrix}
	1.02 & -0.1 \\ 0.1 & 0.98
	\end{bmatrix},\qquad
	B = \begin{bmatrix}
	0.1 & 0 \\ 0.05 & 0.01
	\end{bmatrix},\qquad
	D = \begin{bmatrix}
	0.01 & 0 \\ 0 & 0.01
	\end{bmatrix}.
\end{equation}
The initial condition is set to
$x_0 = \begin{bmatrix}
-0.3 & 1.2\end{bmatrix}^\top$.
Notice that the eigenvalues of the $A$ matrix lie outside the unit disk ($\lambda_{1,2} = 1.0 \pm \mathrm{i}\ 0.098$).
In addition, we consider the following state chance constraint
\begin{align}\label{eq:SimpleExStateConst}
	\Pr\left(\begin{bmatrix}
	-2 & 1
	\end{bmatrix} x_k \leq 2.5\right) \geq 1 - 10^{-3},
\end{align}
for all $k = 0, 1,\ldots$.
We wish to minimize the cost function in~(\ref{prob:InfiniteHorizonOptimalControlProblem}) with the following matrices
\begin{align}\label{eq:SimpleExQR}
	Q = \begin{bmatrix}
		2 & 0 \\ 0 & 1
	\end{bmatrix},\qquad
	R = \begin{bmatrix}
		5 & 0 \\ 0 & 20
	\end{bmatrix},
\end{align}
while satisfying the constraints (\ref{eq:SimpleExStateConst}).

Figure~\subref*{fig:SimpleExUncontrolled} shows 100 sample trajectories of the evolution of the uncontrolled system.
The trajectories follow increasingly large spiral paths that violate the constraint~(\ref{eq:SimpleExStateConst}).
Figure~\subref*{fig:SimpleExLQR} shows the results of 100 sample trajectories using a controller with the infinite-horizon LQR gain corresponding to~(\ref{eq:SimpleExQR}).
As LQR controllers do not take into account any constraints, the majority of the trajectories in Fig.~\subref*{fig:SimpleExLQR} also violate the state constraint~(\ref{eq:SimpleExStateConst}).
\begin{figure}[htbp]
	\centering
	\subfloat[Uncontrolled trajectories.\label{fig:SimpleExUncontrolled}]{\includegraphics[width=0.5\columnwidth]{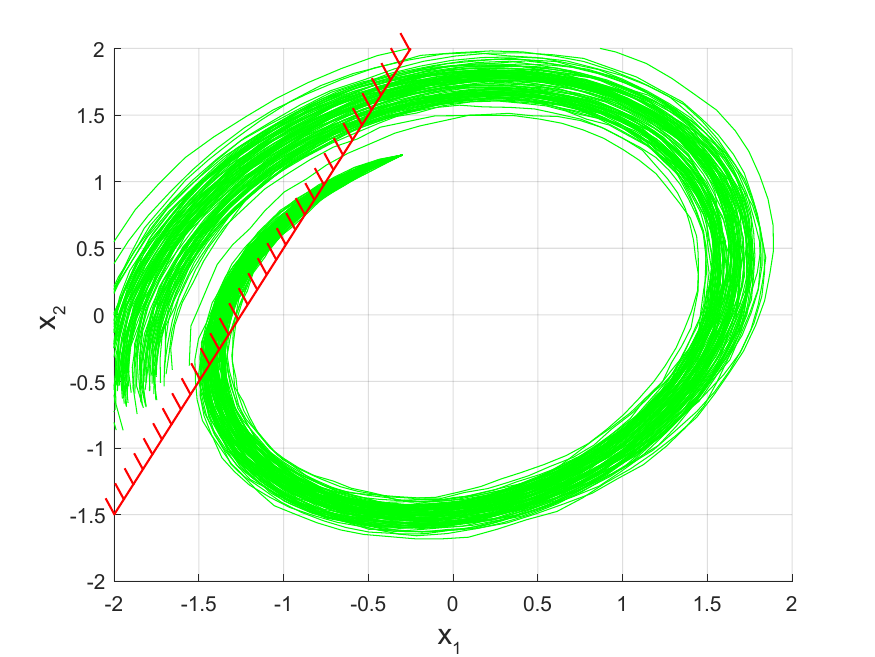}}
	\subfloat[Trajectories from LQR controller. \label{fig:SimpleExLQR}]{\includegraphics[width=0.5\columnwidth]{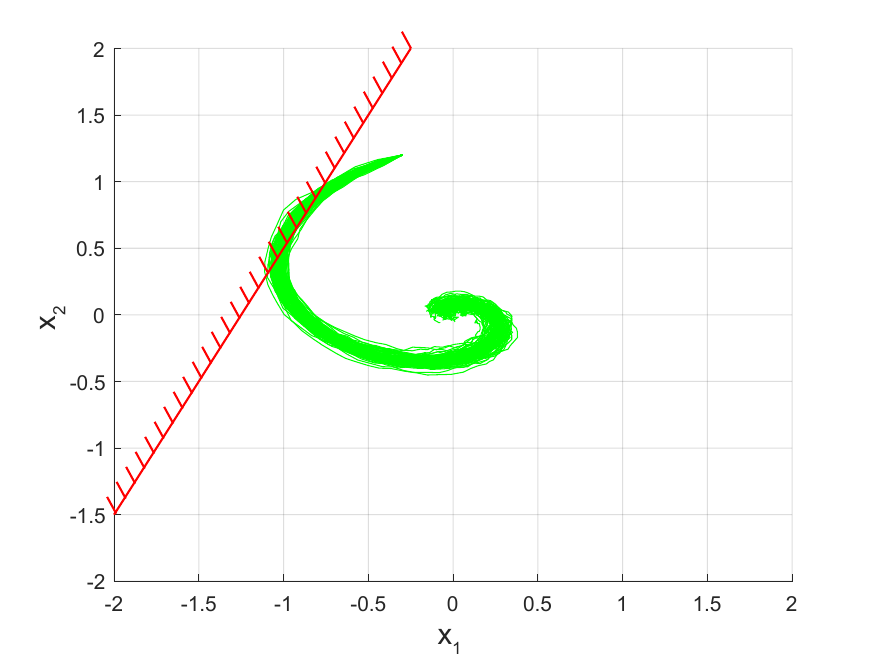}}
	\caption{System state trajectories.}
\end{figure}

We first apply the SMPC approach proposed in~\cite{farina2013probabilistic}.
However, direct application of the methodology in~\cite{farina2013probabilistic} led to an inaccurate estimate of the state covariance at each time step, and hence, to a difficulty satisfying the chance constraints.
Consequently, we modified the approach in~\cite{farina2013probabilistic} as follows.

The terminal cost in~\cite{farina2013probabilistic} is
\begin{align}
	\min_{u_{k|k},u_{k+1|k},\ldots, u_{k+N-1|k}} &J_{N}(x_k; u_{k|k},u_{k+1|k},\ldots, u_{k+N-1|k}) =  \nonumber \\
	&\mathbb{E}_k\left[x_{k+N|k} Q_N x_{k+N|k} + \sum_{t=k}^{k+N-1}x_{t|k}^\top Q x_{t|k} + u_{t|k}^\top R u_{t|k}\right],
\end{align}
where $Q_N$ is the solution of the following discrete-time algebraic Riccati equation
\begin{align}
	A^\top Q_N A - Q_N - A^\top Q_N B (B^\top Q_N B + R)^{-1} B^\top Q_N A + Q = 0.
\end{align}
In~\cite{farina2013probabilistic} the covariance at the end of the horizon is bounded from above by the solution of the discrete-time Lyapunov equation~(\ref{eq:FarinaTerminalCond}).
The terminal mean set  $\bar{\mathcal{X}}^\mu_f$ is the positive invariant set such that
\begin{align}
	(A+BK_{\rm LQR}) \mu \in \bar{\mathcal{X}}^\mu_f,\qquad \forall\ \mu \in \bar{\mathcal{X}}^\mu_f,
\end{align}
where $K_{\rm LQR}$ is the LQR controller gain.
The control policy in~\cite{farina2013probabilistic} uses feedback of the state deviation from the mean, which leads to the following covariance dynamics
\begin{align}\label{eq:covDynamicsFarina}
	\Sigma_{t+1|k} = (A+BK_{t|k})\Sigma_{t|k}(A+BK_{t|k})^\top + DD^\top,
\end{align}
which is a non-convex constraint due to the coupling between $K_{t|k}$ and $\Sigma_{t|k}$.
The authors of~\cite{farina2013probabilistic} mention in~\cite{farina2016model} that they used the following convex relaxation technique proposed in~\cite{primbs2009stochastic} with the mild assumption that $\Sigma_{t|k} \succ 0$ for all $t  \ge k$,
\begin{align}
	\Sigma_{t+1|k} \succeq (A+B\Theta_{t|k})\Sigma_{t|k}^{-1}(A+B\Theta_{t|k})^\top + DD^\top,
\end{align}
which is, using Schur complement, equivalent to the following LMI
\begin{align}
	\begin{bmatrix}
		\Sigma_{t+1|k} & A + B\Theta_{t|k} & D \\
		(A + B\Theta_{t|k})^\top & \Sigma_{t|k} & 0\\
		D^\top & 0 & I_{n_w}
	\end{bmatrix}\succeq 0,
\end{align}
where $\Theta_{t|k} = K_{t|k}\Sigma_{t|k}$ is a new design variable.
However, in this example, we observed that this relaxation led to imprecise computation of the covariance.
So we used instead the disturbance feedback approach of~\cite{oldewurtel2008tractable,paulson2017stochastic}, where the control input is an affine function of the past disturbance sequence
\begin{align}\label{eq:distFeedback}
	u_{t|k} = v_{t|k} + \sum_{\tau = k}^{t - 1} M_{t,\tau}Dw_\tau,
\end{align}
which is known to lead to a convex formulation of the covariance dynamics~\cite{goulart2006optimization}.
After confirming that the problem is feasible with initial condition $\mu_0 = x_0, \Sigma_0 = 0$,
we used the initialization described in Section~\ref{sec:init}   to maintain feasibility for $k \geq 1$.
The horizon was chosen as $N = 10$.

The proposed CS-SMPC algorithm with the same horizon length was also applied to the system.
We used the same terminal target covariance as the one in~(\ref{eq:FarinaTerminalCond}).
The results are shown in Fig.~\ref{fig:SimpleExDistAll} (100 sample trajectories shown).
In both cases the trajectories successfully avoid the constraint and converge to the origin and the overall behavior is very similar.

The main difference between the two methods in this example is the computational cost.
As shown in Fig.~\ref{fig:SimpleExSolverTime}, the CS-SMPC algorithm exhibits faster computational speed.
This superior performance of CS-SMPC  is due to the difference in the control approach formulation.
The CS-SMPC algorithm uses the current value of the $y$ variable, and thus, the $K$ matrix in~(\ref{eq:Kmat}) is block diagonal, while the disturbance feedback controller~(\ref{eq:distFeedback}) uses the past disturbance sequence, implying that a lower block triangular matrix is needed (see~\cite{oldewurtel2008tractable,goulart2006optimization}), which leads to more computations.

\begin{figure}[htbp]
	\centering
	\subfloat[Trajectories resulting from the modified controller in~\cite{farina2013probabilistic}. 
	\label{fig:SimpleExDist}]{\includegraphics[width=0.5\columnwidth]{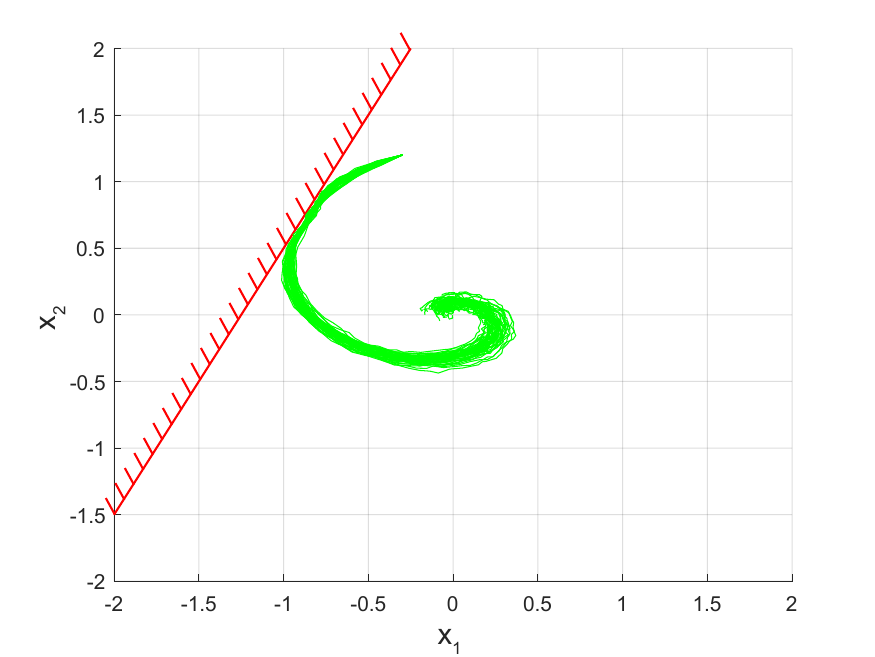}}
	\subfloat[Trajectories resulting from the CS-SMPC approach. \label{fig:SimpleExCSSMPC}]{\includegraphics[width=0.5\columnwidth]{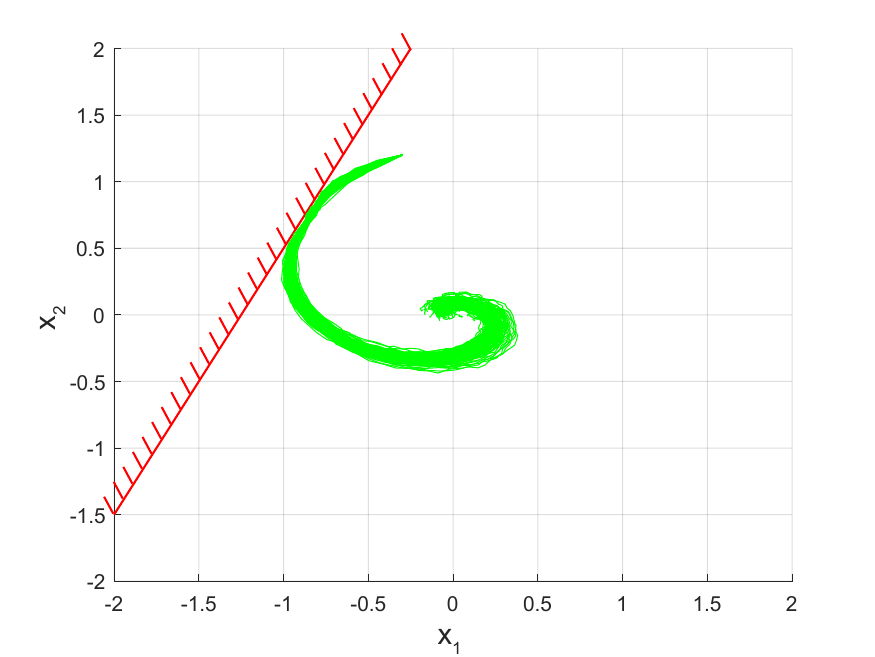}}
	\caption{System state trajectories for Example~1.}
	\label{fig:SimpleExDistAll}
\end{figure}

\begin{figure}[htbp]
	\centering
	\includegraphics[width=0.6\columnwidth]{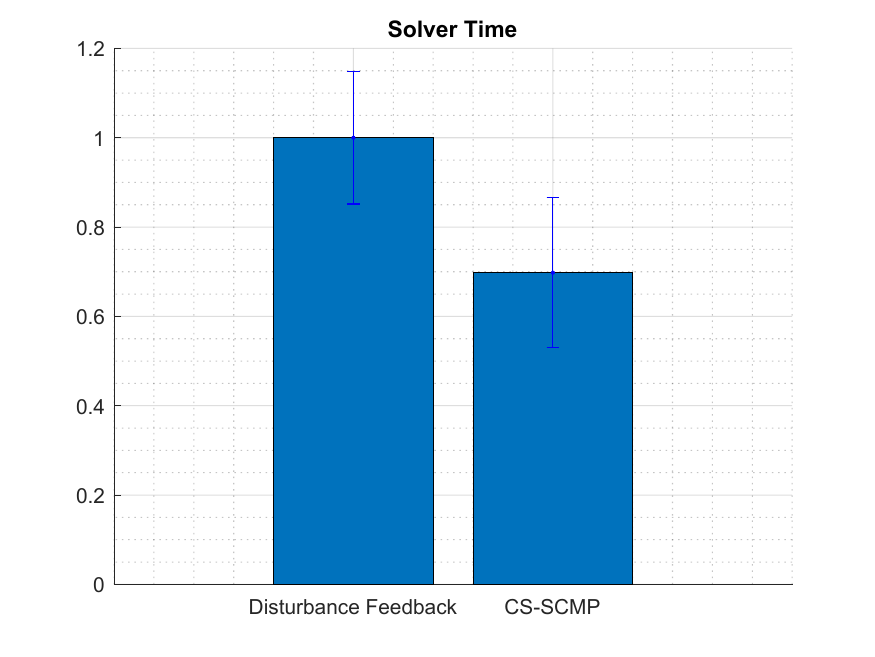}
	\caption{Mean and standard deviation of the computation time of each method. The time is normalized by the computation time of the disturbance feedback method. \label{fig:SimpleExSolverTime}}
\end{figure}

\subsection{Vehicle Control Example}

The previous simple numerical example illustrated the computational benefits of the CS-SMPC approach stemming from the convexity of the problem formulation and the block diagonal structure of the feedback gain matrix.
In this section, we validate the efficacy of the proposed CS-SMPC algorithm via a more realistic example of a racing vehicle driving around a road circuit.

The key benefit of using CS-SMPC for this example is illustrated in Fig.~\ref{fig:RaceCarControl}.
A deterministic MPC approach, as shown in Fig.~\subref*{fig:RaceCar_DMPC}, neglects the effect of stochastic disturbances, and thus, 
a safety margin to the constraint boundaries is needed.
Deterministic MPC thus requires trial-and-error to find reasonable values to achieve good performance while not violating the constraints.
Figure~\subref*{fig:RaceCar_RandomizedMPC} shows an example of a planned trajectory using a stochastic MPC controller with open-loop vehicle dynamics.
Since the effect of noise increases with time, it is difficult to have a long time horizon.
Stochastic tube-MPC uses closed-loop vehicle dynamics as shown in Fig.~\subref*{fig:RaceCarSTMPC}.
As the stabilizing gain of a stochastic tube-MPC is generally constant, the resulting state covariance converges to a constant value.
In addition, a priori calculation of the feedback gains often requires trial and error,
especially in constrained environments,  
Fig.~\subref*{fig:RaceCarCSSMPC} illustrates the benefit of the proposed CS-SMPC approach.
By directly controlling the covariance of the system state, the mean trajectory is steered to the inner edge of the road, which leads to a better performance for a race car trying to minimize lap time. 

\begin{figure}[htbp]
	\captionsetup[subfigure]{width=0.4\columnwidth}
	\centering
	\subfloat[Deterministic MPC.\label{fig:RaceCar_DMPC}]{\includegraphics[width=0.50\columnwidth]{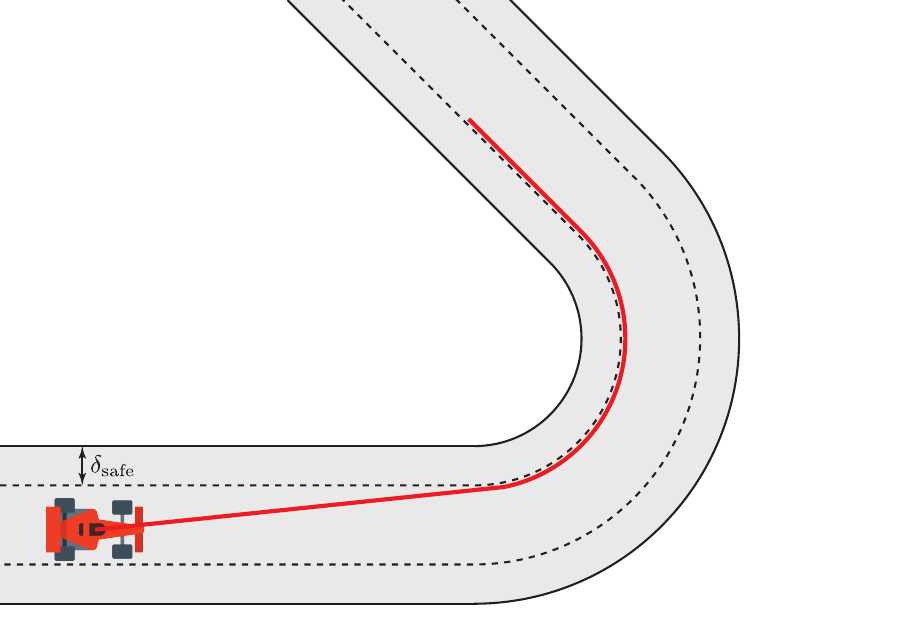}}
	\subfloat[Stochastic MPC with open-loop vehicle dynamics.\label{fig:RaceCar_RandomizedMPC}]{\includegraphics[width=0.50\columnwidth]{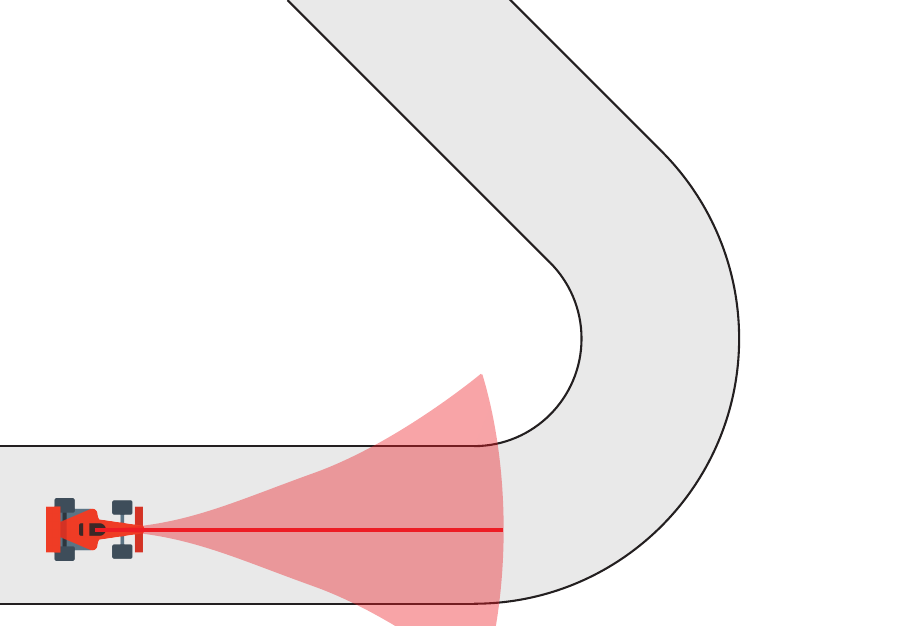}}\\
	\subfloat[Stochastic MPC with closed-loop constant gain vehicle dynamics.\label{fig:RaceCarSTMPC}]{\includegraphics[width=0.50\columnwidth]{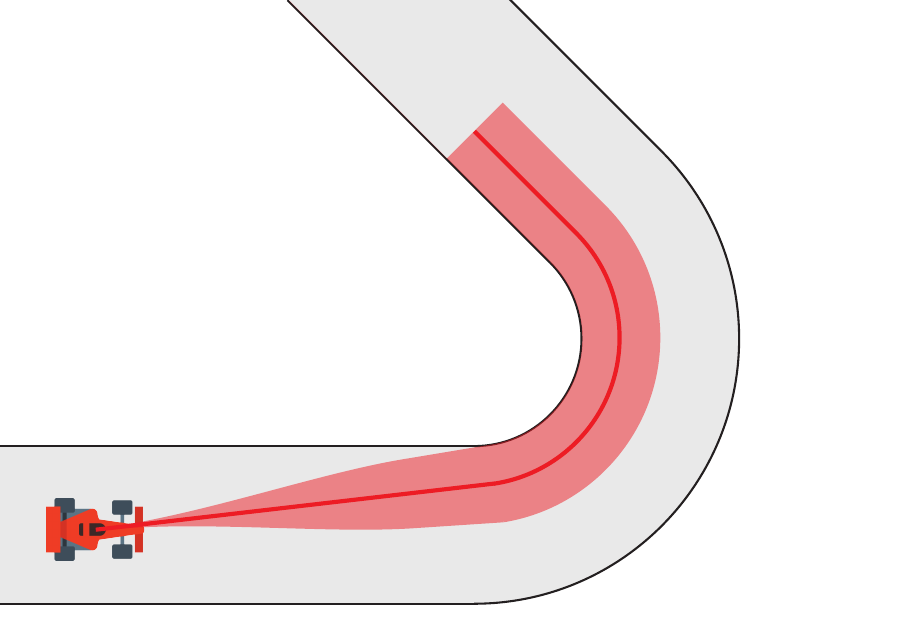}}
	\subfloat[Stochastic MPC with closed-loop time-varying gain vehicle dynamics.\label{fig:RaceCarCSSMPC}]{\includegraphics[width=0.50\columnwidth]{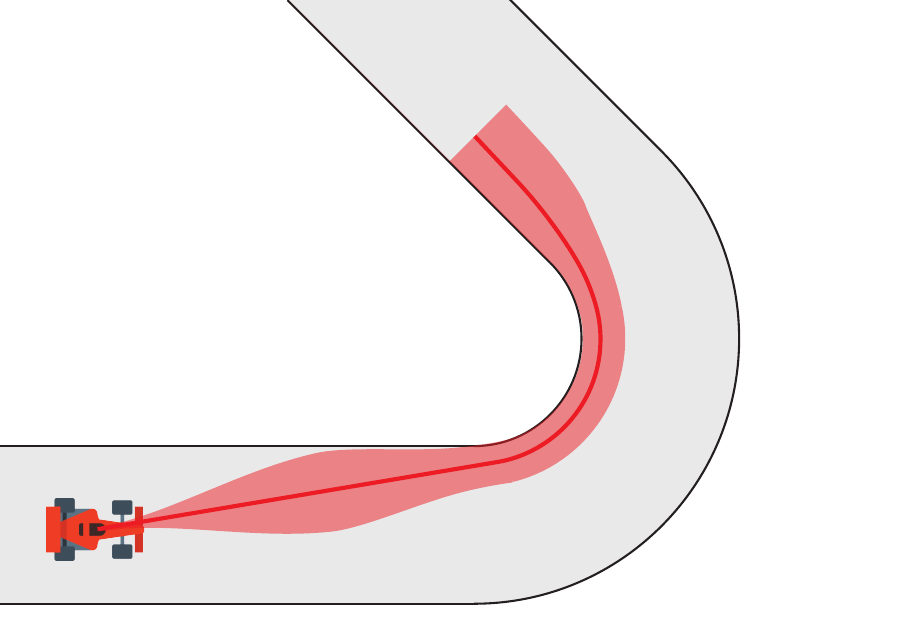}}\\
	\caption{Comparison of MPC approaches for the vehicle control example.
		Each figure shows a planned trajectory for a race car using different MPC approaches. 
		The bold lines indicate the mean trajectories, and the shaded areas represent 1-$\epsilon$ confidence regions.
		By directly controlling the covariance, it is possible to design more aggressive controllers that operate closer to the constraints.\label{fig:RaceCarControl}
	}
\end{figure}

For this example, we use the linearized bicycle model assuming constant longitudinal vehicle speed~\cite{ackermann1993robust} shown in Fig.~\ref{fig:BicycleModel}.
\begin{figure}[htbp]
	\centering
	\includegraphics[width=0.6\columnwidth]{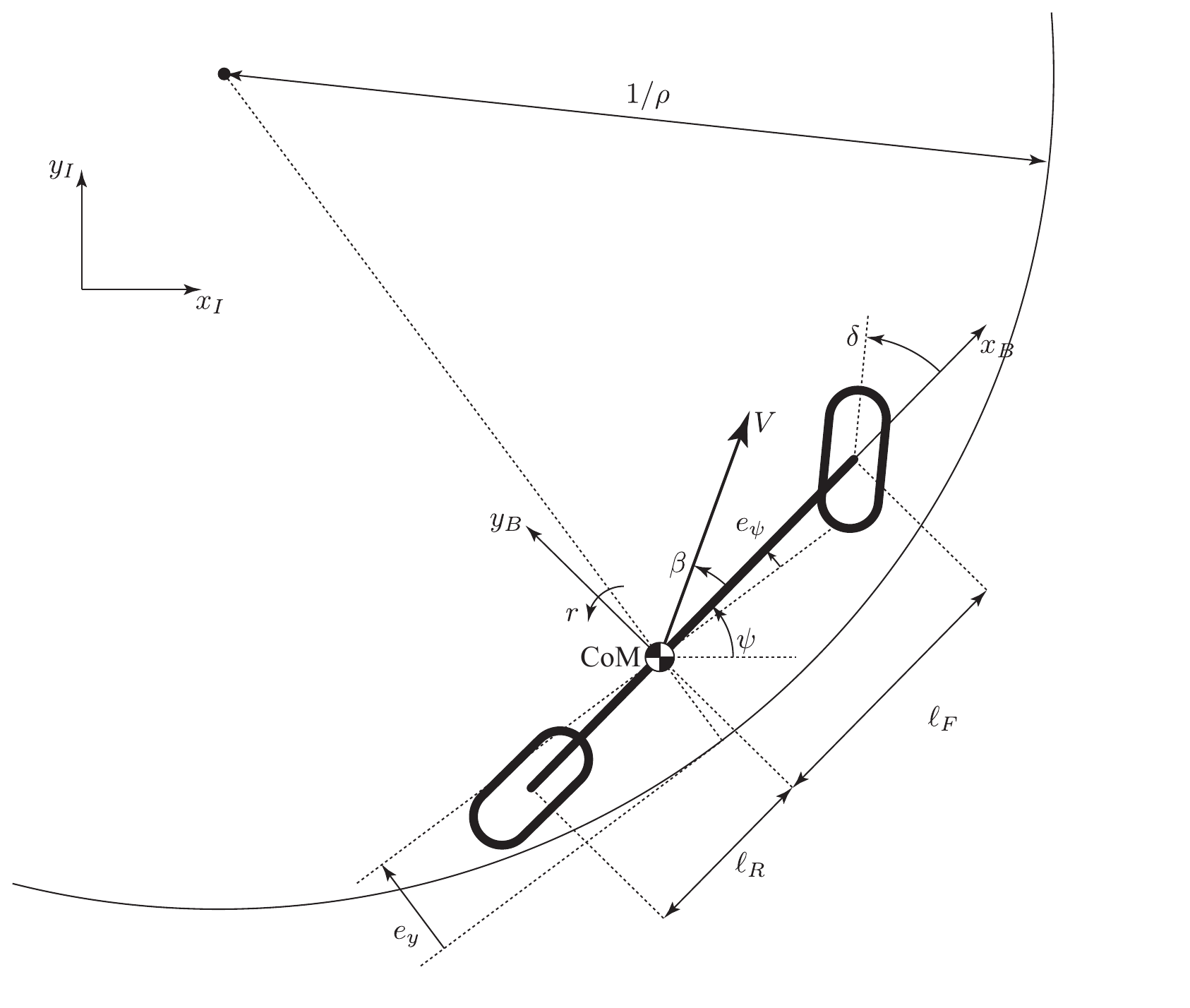}
	\caption{Bicycle model. ($x_I$, $y_I$) is the inertial frame, and ($x_B$, $y_B$) is the body frame.  \label{fig:BicycleModel}}
\end{figure}
The continuous dynamics is described as follows.
\begin{subequations}\label{eq:BicycleModel}
	\begin{align}
		\dot{\beta} &= -\frac{C_r + C_f}{mV_x}\beta + \left(-1 + \frac{\ell_RC_r - \ell_FC_f}{mV_x^2}\right)r + \frac{C_f}{mV_x} \delta,\\
		\dot{r} &= \frac{\ell_RC_r - \ell_FC_f}{I_z}\beta - \frac{\ell_R^2C_r + \ell_F^2C_f}{I_zV_x}r+ \frac{\ell_FC_f}{I_z}\delta,\\
		\dot{e_\psi} &= r - V_x \rho,\\
		\dot{e_y} &= V_x\beta + V_x e_\psi.
	\end{align}
\end{subequations}
The state variables in~(\ref{eq:BicycleModel}) are the side-slip angle $\beta$, the vehicle yaw rate $r$, the heading angle error $e_\psi$, and the lateral deviation error $e_y$.
The inputs to the system are the front wheel angle $\delta$ and the curvature of the road centerline $\rho$, which is a function of the distance along the road centerline.
We assume constant longitudinal velocity.
The system parameters are listed in Table~\ref{table:VehicleParameters} along with the numerical values used in this example.
Note that the vehicle direction angle~$\psi$ can be computed from
\begin{align}
	\dot{\psi} = r.
\end{align}
The vehicle dynamics~(\ref{eq:BicycleModel}) can be represented as an LTI system
\begin{align}
	\dot{x} = A_cx  + B_cu + C_c \rho,
\end{align}
where $x = \begin{bmatrix}
\beta & r & e_\psi & e_y
\end{bmatrix}^\top$ and $u = \delta$.
Using zero-order hold with $\Delta t=0.5$~sec, the discretized LTI dynamics are
\begin{align}\label{eq:NominalVehicleDynamics}
	x_{k+1} = A x_k + B u_k + C \rho_k.
\end{align}
Setting~(\ref{eq:NominalVehicleDynamics}) as the nominal dynamics, our interest is to control the following stochastic dynamics
\begin{align}\label{eq:StochasticVehicleDynamics}
	x_{k+1} = A x_k + B u_k + C\rho_k + Dw_k,
\end{align}
where $D = 0.01I_4$ using the CS-SMPC framework.
The noise term embodies modeling errors as well as disturbances stemming from the interaction of the vehicle wheels with the ground.
The geometry of the road circuit is depicted in Fig.~\ref{fig:RoadGeo}.
The vehicle starts from the origin and drives around the track counter-clockwise.
The state constraint encodes the requirement to keep the vehicle on the road and the system state close enough to the origin.
Specifically, we impose the state constraints
\begin{align}\label{eq:stayInRoadConst}
\begin{bmatrix}
\beta_{\min} \\ r_{\min} \\ e_{\psi,\min} \\ e_{y,\min}
\end{bmatrix} \leq x_k \leq
\begin{bmatrix}
\beta_{\max} \\ r_{\max} \\ e_{\psi,\max} \\ e_{y,\max}
\end{bmatrix},
\end{align}
for all $k \geq 0$.
Notice that, although the road circuit in Fig.~\ref{fig:RoadGeo} is non-convex in the global coordinate frame, the state constraint~(\ref{eq:stayInRoadConst}) is convex.
We set $\beta_{\max} = 0.1$ rad, $\beta_{\min} = -0.1$ rad, $r_{\min} = -1.5$~rad/s, $r_{\max} = 1.5$~rad/s, $e_{\psi,\min} = -0.5$~rad, $e_{\psi,\max} = 0.5$~rad, $e_{y,\min} = -2$~m, and $e_{y,\max} = 2$~m.
In addition, the steering wheel angle is restricted to
\begin{align}
	\delta_{\min} \leq \delta_k \leq \delta_{\max},
\end{align}
for all $k \geq 0$.
In this work we set $\delta_{\min} = -0.25$~rad and $\delta_{\max} = 0.25$~rad.
The state and input constraints were formulated in terms of chance constraints as in (\ref{eq:InfinitecontrolSC}) and (\ref{eq:InfinitecontrolCC}) with $p_{x,i} = p_{u,j} = 10^{-3}$.
The length of the horizon was set to $N = 8$, which corresponds to an actual time horizon of 4 sec.
The cost matrices were set as
\begin{align}\label{eq:VehiceExQandR}
	Q = \mathtt{blkdiag}(10^{-2},0,10^{-2},10^{-8}),\qquad R = 1.
\end{align}
We chose these values so that the vehicle fully utilizes the width of the road, while minimizing the control energy.
Note that, unlike the previous example in Section~\ref{subsec:NSSimpleExample}, the approach in~\cite{farina2013probabilistic} does not work for this scenario, because the terminal covariance in~(\ref{eq:FarinaTerminalCond}) becomes
\begin{align}
\Sigma_f = \begin{bmatrix}
 0.0001 &  -0.0000 &   0.0000 &   0.0002\\
-0.0000 &   0.0001 &  -0.0001 &  -0.0072\\
0.0000  & -0.0001  &  0.0005  & -0.0003\\
0.0002  & -0.0072  & -0.0003  & 26.9796
\end{bmatrix},
\end{align}
and the variance of $e_y$ is too large to satisfy the constraint~(\ref{eq:stayInRoadConst}).
In order to satisfy this constraint we need to modify $\Sigma_f$.
Since $\Sigma_f$ in~(\ref{eq:FarinaTerminalCond}) is an implicit function of the $Q$ and $R$ matrices, 
 we need to tune
 $Q$ and/or $R$ weight matrices by trial and error  till a suitable values for $\Sigma_f$ is found.
 For this problem, it was found that modifying the (4,4) element of $Q$ results in a solution for $\Sigma_f$
that eventually makes the vehicle stay closer to the centerline of the road.

\begin{table}
	\centering
	\caption{Vehicle parameters and values. (CoM: the center of mass of the vehicle).\label{table:VehicleParameters}}
	\begin{tabular}{|c|c|c|}
		\hline
		Notation & Meaning & Used numerical value \\\hline
		\hline
		$m$& Vehicle mass &  1653 kg \\\hline
		$I_z$& Vehicle yaw inertia  & 2765 $\mathrm{kgm}^2$  \\\hline
		$V_x$& Longitudinal velocity & 15 m/s  \\\hline
		$\ell_F$& Distance from CoM to the front axle & 1.402 m  \\\hline
		$\ell_R$& Distance from CoM to the rear axle & 1.646 m  \\\hline
		$C_f$& Front tire cornering stiffness  & 42 kN/rad  \\\hline
		$C_r$& Rear tire cornering stiffness & 81 kN/rad \\\hline
	\end{tabular}
\end{table}

\begin{figure}
	\centering
	\includegraphics[width=0.7\columnwidth]{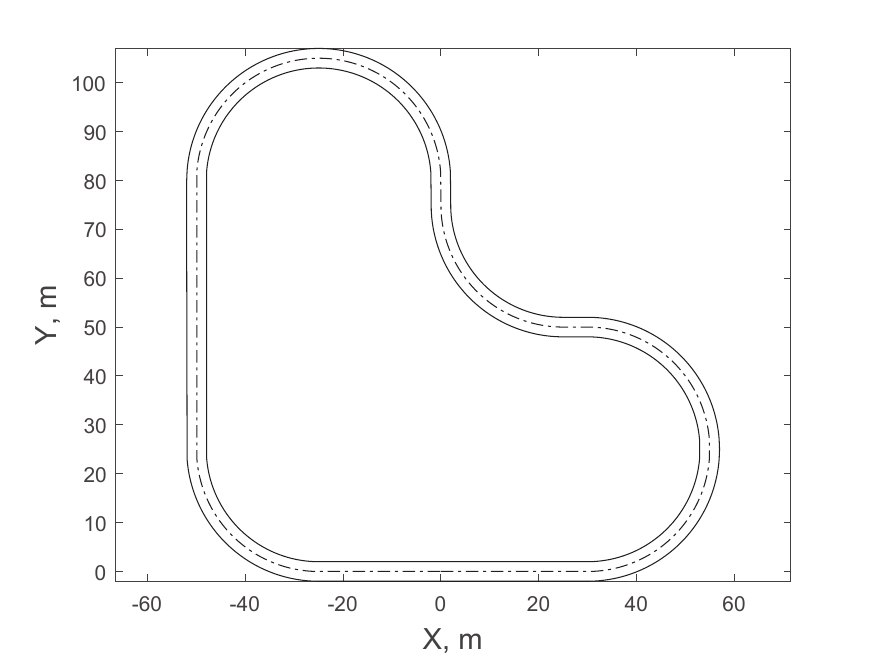}
	\caption{Geometry of the road circuit.\label{fig:RoadGeo}}
\end{figure}

Next, we present the result with the proposed CS-SMPC approach.
In order to determine the terminal covariance, we first solve the following problem to obtain an assignable covariance.
\begin{subequations}
	\begin{align}
		\minimize &\ \|\Sigma_f - \Sigma^d_f\|_F,\\
		\textrm{subject to} &\ (\ref{eq:AssignableConds}),
	\end{align}
\end{subequations}
where $\Sigma^d_f$ is a desired terminal covariance computed as the terminal covariance when the system is controlled by a stabilizing controller.
For our example an LQR controller was chosen with values of $Q$ and $R$ given in~(\ref{eq:VehiceExQandR}).
Specifically, the covariance dynamics is computed by
\begin{equation}
	\Sigma_{t+1|k} = (A+BK_{\rm LQR})\Sigma_{t|k}(A+BK_{\rm LQR})^\top + DD^\top,\qquad
	\Sigma_{k|k} = 0,
\end{equation}
where we set $\Sigma^d_f = \Sigma_{k+N|k}$.
In this example, the values of $\Sigma^d_f$ and $\Sigma_f$ were computed as follows
\begin{align*}
\Sigma^d_f = \begin{bmatrix}
	 0.0001 &  -0.0000 &   0.0000 &   0.0001\\
	-0.0000 &   0.0001 &  -0.0001 &  -0.0026\\
	0.0000  & -0.0001  &  0.0004  &  0.0087\\
	0.0001  & -0.0026  &  0.0087  &  0.3595
\end{bmatrix},\quad
\Sigma_f = \begin{bmatrix}
	0.0001 &  -0.0000  &  0.0000  &  0.0001\\
	-0.0000&    0.0002 &  -0.0001 &  -0.0023\\
	0.0000 &  -0.0001  &  0.0002  & -0.0002\\
	0.0001 &  -0.0023  & -0.0002  &  0.3640
\end{bmatrix}.
\end{align*}
Using this value of~$\Sigma_f$, we computed $\tilde{K}$ from~(\ref{eq:Ktilde}) and $P_{\rm mean}$ from~(\ref{eq:meanCond}).
Figure~\ref{fig:CSMPC} shows 100 sample trajectories controlled by the CS-SMPC algorithm.
The vehicle successfully satisfies the constraints and stays on the road despite the stochastic disturbance.

Contrary to the previous trial and error approach to specify the suitable terminal covariance, 
the proposed CS-SMPC approach allows us to \textit{directly} shape $\Sigma_f$ so that the state satisfies the probabilistic constraints at the end of the horizon.
A consequence is that we can choose  $\Sigma_f$ so that the mean state of the vehicle state operates closer to the road boundaries (see also Fig.~\subref*{fig:RaceCarCSSMPC}),
thus taking full advantage of the available operational region.

\begin{figure}
	\centering
	\includegraphics[width=0.7\columnwidth]{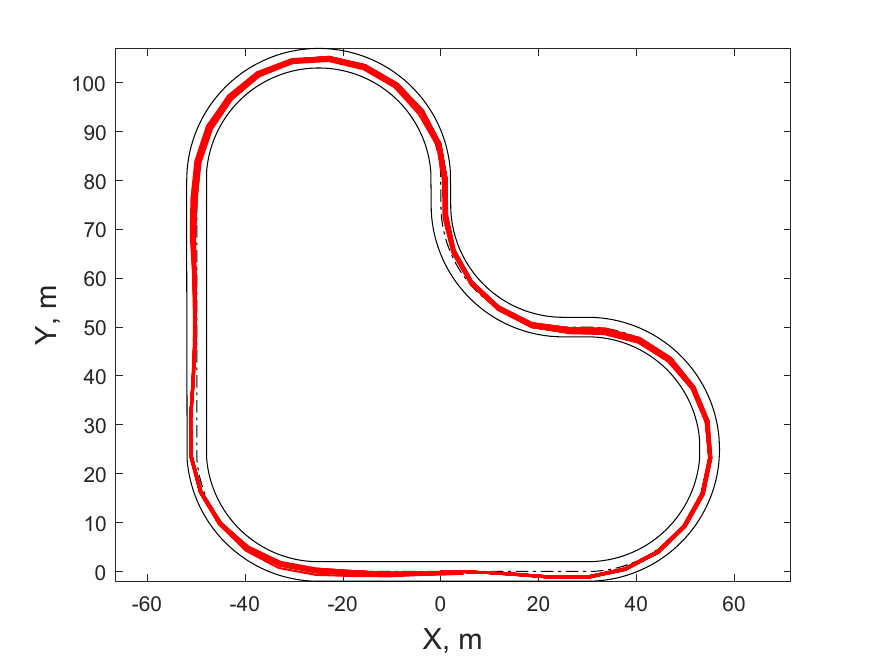}
	\caption{100 sample trajectories controlled by CS-SMPC approach.\label{fig:CSMPC}}
\end{figure}

We also compared CS-SMPC against a deterministic MPC controller to illustrate the benefits of a stochastic problem formulation.
Specifically, we ignore the additive noise in~(\ref{eq:StochasticVehicleDynamics}) at each time step, and minimize the following quadratic cost
\begin{align}
	J = \sum_{k=0}^N (x_k^\top Q x_k + u_k^\top R u_k) + x_N^\top Q_p x_N,
\end{align}
where $Q_p$ is the solution of the following discrete-time algebraic Riccati equation
\begin{align}
	A^\top Q_p A - A^\top Q_p B(B^\top Q_p B + R)^{-1}B^\top Q_p A + Q = Q_p.
\end{align}
The initial condition of the state was set to zero.
The resulting trajectory without noise is depicted in Fig.~\ref{fig:DetMPC}.
However, when noise is added to the system, this controller cannot satisfy the constraints as the vehicle gets too close to the inner edge of the road.
This is expected since the deterministic MPC controller does not consider the added disturbance.

\begin{figure}
	\centering
	\includegraphics[width=0.7\columnwidth]{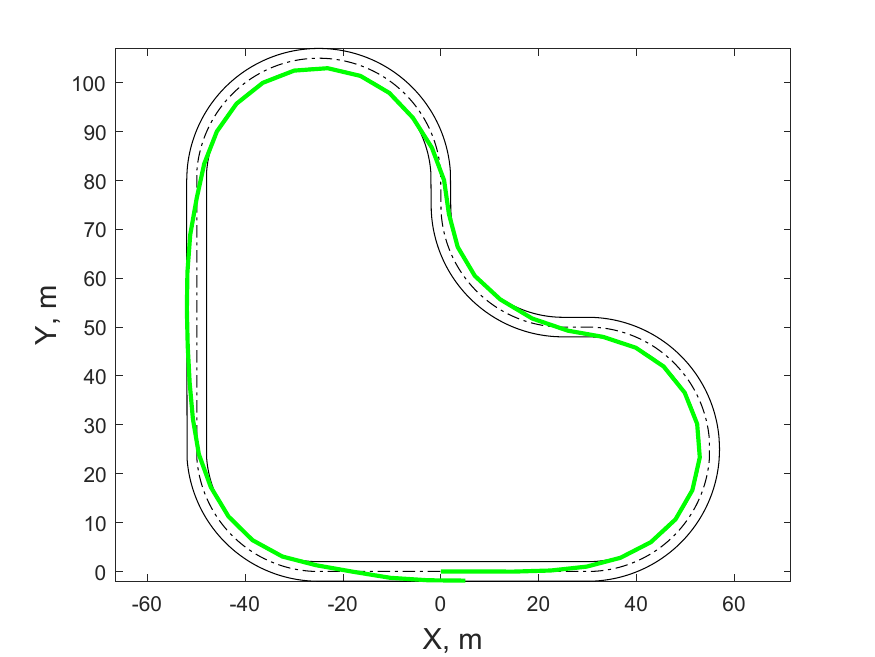}
	\caption{Result of Deterministic MPC without noise.\label{fig:DetMPC}}
\end{figure}

\section{Conclusions}  \label{sec:Summary}

In this paper, we introduced a novel stochastic model predictive control scheme for constrained linear systems with additive Gaussian noise.
The proposed approach makes use of the recently developed finite horizon optimal covariance steering theory, which converts the original stochastic optimal control problem at each iteration of the MPC algorithm to a deterministic convex programming problem.
We showed that the CS-SMPC approach ensures recursive feasibility and guaranteed stability.
In contrast to previous robust and stochastic MPC approaches that guarantee recursive feasibility assuming that the disturbances lie in a compact set, the proposed CS-SMPC approach guarantees this property by constraining the maximal terminal covariance instead.
By doing so, we are able to deal with unbounded additive noise disturbance.
In addition, via numerical simulations, we showed that using covariance steering to compute the future state covariance is more accurate 
and computationally more efficient than previously proposed approaches in the literature.

It is worth noting that since LTI systems are considered, the mean and covariance equations hold even when the noise  
has a non-Gaussian distribution. 
Of course, in the non-Gaussian case the true state distribution can become much more complicated, 
and that this would affect the state constraint implementation. 
However, the rest of the CS-SMPC analysis remains virtually the same.

One drawback of setting the maximal terminal covariance is the need to use semidefinite programming (SDP) to solve the relevant optimization problem.
Solving SDP problems is, in general, computationally more involved than solving a linear program (LP) or a quadratic program (QP), which are used in the majority of MPC algorithms.
However, efficient algorithms do exist for solving SDP such as~\cite{helmberg1996interior,yamashita2012primal}, hence the proposed CS-SMPC approach is still an attractive alternative to existing SMPC methods.

\vspace{20pt}
\textbf{Acknowledgment:}
This work has been supported by ARO award W911NF-16-1-0390 and NSF award CPS-1544814.
The first author has also received partial support from the Funai Foundation for Information Technology.
The authors would also like the thank the anonymous reviewers for their excellent suggestions to improve this paper and for also pointing out the connections 
of CS-SMPC with the work of~\cite{farina2013probabilistic}.

\bibliographystyle{IEEEtran}
\bibliography{CSMPC,InputConstrainedCS}

\end{document}